\documentclass[11pt]{amsart}
\usepackage{pdfsync}
\usepackage{fullpage}
\usepackage{amsmath}
\usepackage{amssymb}
\usepackage{MnSymbol}
\usepackage{verbatim}
\usepackage[linktocpage]{hyperref}
\usepackage{enumitem}
\usepackage{hyperref}

\newcommand{\real}{{\mathbb R}}

\newcommand{\vecw}{{\boldsymbol w}}

\def\Hom{\mathop{\rm Hom}\nolimits}

\def\Re{\mathop{\rm Re}\nolimits}

\def\Sym{\mathop{\rm Sym}\nolimits}

\def\tr{\mathop{\rm tr}\nolimits}

\def\diag{\mathop{\rm diag}\nolimits}

\newtheorem{df}{Definition}[section]
\newtheorem{thm}[df]{Theorem}

\newtheorem{prop}[df]{Proposition}
\newtheorem{rem}[df]{Remark}

\newtheorem{lem}[df]{Lemma}

\newtheorem{cor}[df]{Corollary}

\newcommand{\ti}{\tilde}
\newcommand{\mc}{\mathcal}
\newcommand{\mb}{\mathbb}

\newcommand{\eq}{eqnarray*}

\title{Bounded differentials on unit disk and the associated geometry}
\author{Song Dai\textsuperscript{1}}

\address{Song Dai\\
Center for Applied Mathematics\\
Tianjin University\\
No.92 Weijinlu Nankai District\\
Tianjin\\
P.R.China 300072}

\email{song.dai@tju.edu.cn}

\author{Qiongling Li\textsuperscript{2}}

\address{Qiongling Li\\
Chern Institute of Mathematics and LPMC\\
Nankai University\\
No. 94 Weijinlu Nankai District\\
Tianjin\\
P.R.China 300071}

\email{qiongling.li@nankai.edu.cn}
\date{}
\begin{document}
\maketitle
\begin{abstract}
For a harmonic diffeomorphism between the Poincar\'{e} disks, Wan showed the equivalence between the boundedness of the Hopf differential and the quasi-conformality.
In this paper, we will generalize this result from quadratic differentials to $r$-differetials. We study the relationship between bounded holomorphic $r$-differentials and the induced curvature of the associated harmonic maps from the unit disk to the symmetric space $SL(r,\mathbb R)/SO(r)$ arising from cyclic/subcyclic harmonic Higgs bundles. Also, we show the equivalences between the boundedness of holomorphic differentials and having a negative upper bound of the induced curvature on hyperbolic affine spheres in $\mathbb{R}^3$, maximal surfaces in $\mathbb{H}^{2,n}$ and $J$-holomorphic curves in $\mathbb{H}^{4,2}$ respectively. Benoist-Hulin and Labourie-Toulisse have previously obtained some of these equivalences using different methods.


\end{abstract}

\tableofcontents

\section{Introduction}
Consider a harmonic diffeomorphism between the unit disk $\mathbb D$ equipped with the hyperbolic metric $g_{\mathbb D}$, by the work of Wan \cite{WAN}, the followings are equivalent: (i) the harmonic map is quasi-conformal; (ii) the energy density is bounded; (iii) its Hopf differential is bounded with respect to $g_{\mathbb D}$.

Wan's result on holomorphic quadratic differentials can be generalized to holomorphic cubic differentials by Benoist and Hulin \cite{BenoistHulin}. There they show that for a hyperbolic affine sphere in $\mathbb R^3$ whose affine metric has conformal type as $\mathbb D$, the followings are equivalent: (i) its affine metric has curvature bounded above by a negative constant; (ii) the affine metric is conformally bounded with respect to $g_{\mathbb D}$; (iii) its Pick differential is bounded with respect to $g_{\mathbb D}$.

The associated equations in these two settings belong to a type of single vortex equation and can be dealt simultaneously. The single vortex equation can be fit into a more general story: the nonabelian Hodge correspondence for Higgs bundles. 

Recall that a Higgs bundle on $\mathbb D$ is a pair consisting of a holomorphic vector bundle $E$ over $\mathbb D$ and a Higgs field as a $End(E)$-valued holomorphic 1-form.
We call a Higgs bundle $(E,\phi)$ has bounded spectral curve if the holomorphic differentials $\tr(\phi), \tr(\phi^{\otimes 2}),\cdots, \tr(\phi^{\otimes r})$ are bounded with respect to the hyperbolic metric $g_{\mathbb D}$. A Hermitian metric $h$ on $E$ is called harmonic if it satisfies the Hitchin equation. A harmonic metric $h$ induces a harmonic map from the unit disk to $GL(r,\mathbb C)/U(r)$ whose energy density is $|\phi|^2_{h,g_{\mathbb D}}g_{\mb{D}}$.

Similar to the compact surface case, one can define the Hitchin section consisting of Higgs bundles over $\mathbb D$ parametrized by a tuple of differentials $q_2, \cdots, q_r$. Among such Higgs bundles, there is a special section called the Hitchin section. And the image of $s(0,\cdots, 0, q_r)/s(0,\cdots, 0, q_{r-1},0)$ are called cyclic/subcyclic Higgs bundles in the Hitchin section. In particular, the two cases discussed at the beginning belong to the case of cyclic Higgs bundles in the Hitchin section for $r=2,3$. For a Higgs bundle over $\mathbb D$, the second author and Mochizuki in \cite{LiMochizuki} showed that a Higgs bundle has bounded spectral curve if and only if the energy density is bounded. This can be viewed as a generalization of rank two and three cyclic Higgs bundles on the equivalence between (ii) and (iii).

Similar to the case for quadratic differentials and cubic differentials, we want to ask if the following is true: for a Higgs bundle in the Hitchin section, by choosing an appropriate harmonic metric,  it has bounded spectral curve if and only if the induced curvature of the harmonic map is bounded above by a negative constant. Note that this question restricts to the negative curvature conjecture for compact hyperbolic surface case in \cite{DaiLi16}. We will approach this question for the case of cyclic and subcyclic Higgs bundles in the Hitchin section.

For a general $r\geq 2$, given a holomorphic $r$-differential $q$ on $\mathbb D$, the second author and Mochizuki \cite{LiMochizuki} showed there is a unique complete real solution on the cyclic rank $r$ Higgs bundle $s(0,\cdots, 0, q)$. Later, using similar techniques, given a holomorphic $(r-1)$-differential $q$ on $\mathbb D$, Sagman in \cite{Sagman} showed there is a unique complete solution on the subcyclic rank $r$ Higgs bundle $s(0,\cdots, 0, q,0)$. Therefore, one can associate to a holomorphic $r$-differential or a holomorphic $(r-1)$-differential a harmonic immersion $f$ from $\mathbb D$ to the symmetric space $N:=SL(r,\mathbb R)/SO(r)$ arising from the complete solution. For both cases, the sectional curvature $K_{\sigma}^N$ of the tangent plane $\sigma$ of $f$ satisfies $K_{\sigma}^N<0$ for $r\geq 3$. By the Gauss equation, the induced curvature is negative.

We would like to investigate the relation between bounded $r$-differentials or $(r-1)$-differentials with the geometry of harmonic maps for general $r\geq 3$. Our main result is the following theorem.
\begin{thm}[part of Theorem \ref{cyclicbound}]\label{Intro:cyclicbound}
Suppose $f: \mathbb D\rightarrow N=SL(r,\mathbb R)/SO(r)$  is the harmonic map induced by a holomorphic $r$-differential $q$ or a holomorphic $(r-1)$-differential $q$ arising from the complete solution for $r\geq 3$. Then the followings are equivalent:
\begin{enumerate}
\item $q$ is bounded with respect to $g_{\mathbb D}$;
\item the conformal harmonic map $f:\mathbb D\rightarrow \Sigma\subset N$ is biLipschitz;
\item there is a constant $\delta>0$ such that $K_{\sigma}^N\leq -\delta$ for every tangent plane $\sigma$ of $\Sigma$;
\item the induced curvature on $\Sigma$ is bounded from above by a negative constant.
\end{enumerate}
\end{thm}

The equivalence between (1) and (2) is already shown in Li-Mochizuki \cite{LiMochizuki}.

\begin{rem}
It is clear that the result still holds if one replaces the unit disk $\mathbb D$ with any hyperbolic Riemann surface.
\end{rem}

We also have a version of Theorem \ref{Intro:cyclicbound} dealing with the equation for the case $r=2,3$: Proposition \ref{n=1 bd}.
When $r=2$, Proposition \ref{n=1 bd} reduces to Wan's result on harmonic diffeomorphisms between $\mathbb D$.

In fact, the main technique in our proof of Proposition \ref{n=1 bd} is adapted from Wan's proof. Wan's result is also generalized to harmonic diffeomorphisms between pinched Hadamard spaces by \cite{LiTW}.  Based on Wan's proof, we introduce the key Lemma \ref{Wan} and \ref{KEYLEMMA}, which will be used frequently in the paper. As a direct corollary of Proposition \ref{n=1 bd} for $r=3$ case, we reprove the theorem of Benoist-Hulin \cite{BenoistHulin} on hyperbolic affine spheres in $\mathbb R^3$, showing the affine metric has curvature bounded above by a negative constant if and only if its Pick differential is bounded with respect to the hyperbolic metric, using a different proof.
The details will be discussed in Section \ref{r=2} and \ref{r=3}.

Holomorphic quartic differentials and sextic differentials appear in the structure data of immersed surfaces in pseudo-hyperbolic spaces. We prove similar results as the case of quardratic and cubic differentials.

For a space-like maximal surface in $\mathbb H^{2,n}$, one can associate a holomorphic quartic differential. By \cite{CTT}, such a maximal surface corresponds to a conformal $SO_0(2,n+1)$-Higgs bundle over a domain in $\mathbb C$ together with a real harmonic metric compatible with the group structure. A maximal surface is called complete if its induced metric is complete. In Section \ref{MaximalSurface}, we prove the following theorem, which is already shown in Labourie-Toulisse \cite{LT} except (i)$\Rightarrow$(ii).

\begin{thm}
Let $X$ be a complete maximal surface in $\mathbb H^{2,n}$. Then  its induced curvature is either negative or constantly zero. In the latter case, $X$ is conformal to the complex plane.

Moreover, suppose $X$ is conformal to $\mathbb D$, the followings are equivalent:\\
(i) the quartic differential is bounded with respect to the hyperbolic metric;\\
(ii) the induced metric is conformally bounded with respect to the hyperbolic metric;\\
(iii) the induced metric has curvature bounded above by a negative constant.
\end{thm}
\begin{rem}
Our technique is different from the one in \cite{LT}. In fact, the method in \cite{LT} is in the similar spirit of the one in \cite{BenoistHulin} for affine spheres and cubic differentials.
\end{rem}

For a space-like $J$-holomorphic curve in $\mathbb H^{4,2}$ such that its second fundamental form never vanishes and has timelike osculation line, one can associate a holomorphic sextic differential $q_6$. By \cite{NieXin}, such a $J$-holomorphic curve corresponds to a subcyclic Higgs bundle $s(0,\cdots, 0,q_6, 0)$ together with a harmonic metric $diag(h_1,h_2,h_3,1,h_3^{-1},h_2^{-1},h_1^{-1})$ satisfying $h_1=2h_2h_3.$ In Section \ref{G2}, we prove the following theorem.
\begin{thm}
Let $X$ be a complete space-like $J$-holomorphic curve in $\mathbb H^{4,2}$ such that its second fundamental form never vanishes and has timelike osculation line. Then its induced curvature is either negative or constantly zero. In the latter case, $X$ is conformal to the complex plane.

Moreover, suppose $X$ is conformal to $\mathbb D$, the followings are equivalent:\\
(i) the sextic differential is bounded with respect to the hyperbolic metric;\\
(ii) the induced metric is conformally bounded with respect to the hyperbolic metric;\\
(iii) the induced metric has curvature bounded above by a negative constant.
\end{thm}

For Toda equations and related geometry on complex plane or more general Riemann surfaces, there are plenty of studies on the solution and corresponding geometry. For harmonic maps between surfaces, see e.g \cite{SchoenYau,TeichOfHarmonic,Parabolic,HAN,HTTW, Gupta}. For hyperbolic affine spheres in $\mathbb R^3$,  see e.g \cite{Calabi,Loftin0, Labourie,BenoistHulinFiniteVolume, BenoistHulin, DumasWolf, nie2019poles}.
For maximal surfaces in $\mathbb H^{2,n}$, see \cite{CTT,TamburelliWolf, labourie2020plateau}. For $J$-holomorphic curves in $\mathbb H^{4,2}$ (or equivalently $S^{2,4}$), see \cite{g2Geometry,Parker}. For cyclic Higgs bundles of general rank, see e.g \cite{cyclichiggsaffinetoda,GuestLin, GuestItsLin,MochizukiToda1,MochizukiToda2,DaiLi16,DaiLiCyclic,LiMochizuki1}.

Recall that the universal Teichm\"uller space $\mathcal T(\mathbb D)$ is the space of quasisymmetric homeomorphisms of $S^1$ fixing three points. Combining Wan's result with the work of \cite{LT, Markovic}, there is a bijection between the space of bounded quadratic differentials with the universal Teichm\"uller space. A recent work \cite{LT} by Labourie and Toulisse constructs an analogue of the universal Teichm\"uller space as a subspace of the space of maximal surfaces in $\mathbb H^{2,n}$ which relates with the space of bounded quartic differentials: Let $\mathcal{QS}_n$ be the space of quasisymmetric maps from $\mathbb P(V)$ to $\partial_{\infty}\mathbb H^{2,n}$ equipped with the $C^0$ topology up to the action of $SO(2,n+1)$. The space $\mathcal{QS}_n$ should be viewed as a higher rank analogue of the universal Teichm\"uller space $\mathcal T(\mathbb D)$. Denote by $H_b^0(\mathbb D, K^4)$ the space of bounded quartic differentials on $\mathbb D$. One can naturally construct a map $\mathcal H$ from $\mathcal{QS}_n$ to the product space $\mathcal T(\mathbb D)\times H_b^0(\mathbb D, K^4)$. In the end of Section \ref{MaximalSurface}, we give a short proof of the properness of this map $\mathcal H$ which was asked by Labourie and Toulisse.

\subsection*{Organization of this paper}
In Section \ref{Tools}, we collect some useful tools, especially Lemma \ref{Wan} which is the key tool. We then introduce preliminaries on Higgs bundles and Toda equation in Section \ref{pre}. In Section \ref{General_r}, we prove the main theorem \ref{cyclicbound}. In Section \ref{r=2}, we study the application of Proposition \ref{n=1 bd} to the harmonic map between surfaces and bounded quadratic differentials. In Section \ref{r=3}, we study the application of Proposition \ref{n=1 bd} to the hyperbolic affine sphere and bounded cubic differentials. In Section \ref{MaximalSurface}, we study maximal surfaces in $\mathbb H^{2,n}$ and bounded quartic differentials. We also discuss the analogous universal Teichm\"uller space in Section \ref{MaximalSurface}. In Section \ref{G2}, we study $J$-holomorphic curves in $\mathbb H^{4,2}$ and bounded sextic differentials.

\subsection*{Acknowledgements}
We want to thank Francois Labourie and Jeremy Toulisse for sending us their preprint and helpful discussion. The second author thanks Xin Nie for helpful explanations on the associated $G_2$-geometry.
The first author is supported by NSF of China (No.11871283, No.11971244 and No.12071338). The second author is supported by the Fundamental Research Funds for the Central Universities and Nankai Zhide foundation.

\section{Main tools}\label{Tools}
In this section, we introduce several useful lemmas.
We first recall the following mean-value inequality for proving the key Lemma \ref{Wan}.
\begin{prop}(Mean-value inequality, \cite[Lemma 2.5]{ChoiTreibergs})\label{5}
Let $(M,g)$ be a complete Riemannian manifold of dimension $n$. Let $c\geq 0$, $R_0>0$, $x_0\in M$. Suppose the Poincar\'e and Sobolev inequalities hold with constant $c_p$ and $c_s$ for functions supported in $B_g(x_0,R_0)$. Suppose $\text{Vol}(B_g(x_0,r))\leq c_2 r^n$, $\forall r\leq R_0$ for some constant $c_2>0$. Then there are constants $p_0=p_0(n,c,R_0,c_p,c_s,c_2)>0$ and $C=C(n,c,R_0,c_p,c_s,c_2)>0$ such that given any nonnegative $W^{1,2}$ supersolution $u$ satisfying $\triangle_{g}u\leq c u$ in $B_{g}(x_0,R_0)$ and $p\in (0,p_0)$, there holds
$$\inf\limits_{x\in B_g(x_0,R_0/4)}u(x)\geq C(\int_{B_g(x_0,R_0/2)}u^pdV_g)^{\frac{1}{p}}.$$
In particular, there are constants $C>0$, $0<p<1$ such that
$$u(x_0)\geq C(\int_{B_g(x_0,R_0/2)}u^{p}dV_g)^{\frac{1}{p}}.$$
\end{prop}

The following is the key lemma we will use frequently in this paper. It follows from the proof in Wan \cite[Theorem 13]{WAN} where he showed that bounded quadratic differentials correspond to quasi-conformal maps between the unit disk.
\begin{lem}\label{Wan}
Let $g$ be a Riemannian metric on $\mb{D}$ which is equivalent to the hyperbolic metric $g_{\mb{D}}$, i.e. $C_1^{-1}g\leq g_{\mb{D}}\leq C_1 g$ for some constant $C_1\geq 1$. Suppose the Gaussian curvature $K_g$ of $g$ satisfies $-H\leq K_g\leq0$ for some constant $H>0$. Let $u>0$ be a smooth function on $\mb{D}$ satisfying $-C_2^{-1}K_g\leq u\leq -C_2K_g$ for some constant $C_2\geq 1$.

Suppose $u$ also satisfies $$\triangle_g u\leq cu$$ for some constant $c>0$, then there is a constant $\delta_1=\delta_1(C_1,H,C_2,c)>0$, $\delta_2=\delta_2(C_1,H,C_2,c)>0$ such that $u\geq \delta_1$ and $K_{g}\leq -\delta_2$.

Moreover if we further assume $g$ is conformal to $g_{\mb{D}}$, then the assumption ``$-H\leq K_g\leq0$" above can be weakened to ``$K_g\leq0$", and the constants $\delta_1,\delta_2$ above in the assertion are independent on $H$.
\end{lem}
\begin{proof}
Since $g$ is equivalent to $g_{\mb{D}}$, noticing that $g_{\mb{D}}$ has exponential volume growth and is homogeneous, so there exists constants $\epsilon,\eta,M>0$, only depending on $C_1$, such that $\text{Vol}(B_{g}(x,r))\geq \epsilon e^{\eta r}-M$ for every ball $B_{g}(x,r)$ with respect to the metric $g$. So there exists a constant $R>0$ depending on $C_1$, such that for every $x\in \mb{D}$, there exists $r\in (0,R)$ satisfying $\frac{d^2}{dr^2}\text{Vol}(B_{g}(x,r))\geq 2\pi+1$.
Since $K_g\leq 0$ and $\mathbb{D}$ is simply connected, the exponential map is a diffeomorphism. Then $$\frac{d^2}{dr^2}\text{Vol}(B_{g}(x,r))=\frac{d}{dr}\text{L}(\partial B_{g}(x,r))=\int_{\partial B_{g}(x,r)}k_{g}ds_{g},$$
where $\text{L}$ is the length functional and $k_{g}$ is the geodesic curvature of $\partial B_{g}(x,r)$ with respect to the metric $g$. Then by the Gauss-Bonnet-Chern formula, we have $$\frac{d^2}{dr^2}\text{Vol}(B_{g}(x,r))=2\pi-\int_{B_{g}(x,r)}K_{g}dV_{g}.$$
So for every $x\in \mb{D}$ there is a constant $r\in (0,R)$ such that $\int_{B_{g}(x,r)}-K_{g}dV_{g}\geq 1$. Since $u\geq-C_2^{-1}K_g$, there is a constant $\alpha>0$ depending on $C_2$ such that $\int_{B_{g}(x,r)}udV_{g}\geq \alpha$.

Now we apply Proposition \ref{5} to obtain the $L^p$ estimate of $u$. In fact, let $R_0=2R$, where $R$ is defined above depending on $C_1$. Since $g$ is equivalent to $g_{\mb{D}}$, there is a constant $c_2>0$ depending on $C_1$ such that $\text{Vol}(B_{g}(x,r))\leq c_2 r^n$, $\forall r\leq R_0,\forall x\in\mb{D}$. Since $-H\leq K_g$, from \cite{Li-GA} Theorem 5.9, the Poincar\'e inequality holds for a constant $c_p$ depending on $H$. To see the Sobolev constant, since it is equivalent to the isoperimetric constant and $g$ is equivalent to $g_{\mb{D}}$, 
we see that the Sobolev inequality holds for a constant $c_s$ depending on $C_1$. Therefore from Proposition \ref{5}, we obtain that there are constants $C>0$, $0<p<1$ depending on $C_1,H,C_2,c$ such that for every $x\in \mb{D}$,
$$u(x)\geq C(\int_{B_{g}(x,R)}u^pdV_{g})^{\frac{1}{p}}.$$
Notice that from the assumption $u$ has an upper bound depending on $H,C_2$. Together with the $L^1$ estimate we obtained above, there are constants $C^{\prime},\delta_1>0$ only depending on $C_1,H,C_2,c$ such that
$$u(x)\geq C||u||_{L^p(B_{g}(x,R))}\geq C^{\prime}||u||^{\frac{1}{p}}_{L^1(B_{g}(x,r))}\geq \delta_1.$$
Since $u$ is mutually bounded by $-K_g$, then there is a constant $\delta_2=\delta_2(C_1,H,C_2,c)>0$ such that $K_{g}\leq -\delta_2$.
Moreover, we further assume $g=e^ug_{\mb{D}}$ and only assume $K_{g}\leq 0$. Then $\triangle_g=\frac{1}{e^u}\triangle_{g_{\mb{D}}}$. As the discussion above, we still have $\int_{B_{g}(x,r)}udV_{g}\geq \alpha$. Now we apply Proposition \ref{5} to $$\triangle_{g_{\mb{D}}}u=e^u\triangle_{g}u\leq Cu,$$ where $C$ depends on $C_1,c$. Then $u(x)\geq C(\int_{B_{g_{\mb{D}}}(x,R)}u^pdV_{g_{\mb{D}}})^{\frac{1}{p}}.$ Since $g_{\mb{D}}$ is equivalent to $g$, we still obtain $u(x)\geq C(\int_{B_{g}(x,R)}u^pdV_{g})^{\frac{1}{p}},$ where $C$ only depends on $C_1,C_2,c$. Then we repeat the same procedure and finish the proof.
\end{proof}

The next lemma is the Cheng-Yau maximum principle which allow us to deal with complete Riemann manifolds whose curvature is bounded from below.
\begin{lem}[Cheng-Yau Maximum Principle \cite{ChengYau}]\label{Cheng-Yau}
Suppose $(M,h)$ is a complete manifold with Ricci curvature bounded from below. Let $u$ be a $C^2$-function defined on $M$ such that
$\triangle_h u\geq f(u),$ where $f:\mathbb{R}\rightarrow \mathbb{R}$ is a function. Suppose there is a continuous positive function $g(t):[a,\infty)\rightarrow \mathbb{R}_+$ such that \\
(i) $g$ is non-decreasing; \\
(ii) $\liminf\limits_{t\rightarrow +\infty}\frac{f(t)}{g(t)}>0$;\\
(iii) $\int_a^{\infty}(\int_b^tg(\tau)d\tau)^{-\frac{1}{2}}dt<\infty$, for some $b\geq a$,\\
 then the function $u$ is bounded from above. Moreover, if $f$ is lower semi-continuous, $f(\sup u)\leq 0$.
\end{lem}
In particular, for $\alpha>1$ and a positive constant $c_0$, one can check if $f(t)\geq c_0t^{\alpha}$ for $t$ large enough, $g(t)=t^{\frac{\alpha+1}{2}}$ satisfy the above three conditions (i)(ii)(iii).


\begin{lem}\label{KC}
Let $g_0$ be a complete Riemannian metric on $\Sigma$ with $K_{g_0}\geq -a$ for some constant $a>0$. Let $g=e^ug_{0}$ be a Riemannian metric conformal to $g_0$ with curvature $K_{g}\leq -b$ for some constant $b>0$. Then there is a constant $C>0$ such that $g\leq \frac{a}{b}g_{0}$.
\end{lem}
\begin{proof}
Locally,
$$-b\geq K_{g}=-\frac{2}{e^u\ti{g}_{0}}\partial_z\partial_{\bar{z}}\log (e^u\ti{g}_{0}).$$ Then
$$-\frac{2}{e^u}\triangle_{g_{0}}u+\frac{1}{e^u}K_{g_0}\leq -b.$$ So
$$\triangle_{g_{0}}u\geq \frac{b}{2}e^u+\frac{K_{g_0}}{2}\geq \frac{b}{2}e^u-\frac{a}{2}.$$
Then by the Cheng-Yau maximum principle, $e^u\leq \frac{a}{b}.$
\end{proof}

Combining Lemma \ref{Wan}, \ref{Cheng-Yau} and \ref{KC}, we show the following lemma, which will be used frequently throughout this paper.

\begin{lem}\label{KEYLEMMA}
Suppose $(\Sigma,g)$ is a complete surface with curvature bounded from below. Suppose there exist constants $c,d>0$ such that
\begin{equation}\label{ModelEquation}
    \triangle_{g}K_{g}\geq cK_{g}(K_{g}+d).
\end{equation}
Then $K_{g}<0$ or $K_g\equiv 0$ in which case $\Sigma$ is parabolic.

When $\Sigma$ is hyperbolic and let $g_{hyp}$ be the unique complete conformal hyperbolic metric. If $g\leq Cg_{hyp}$ for some constant $C>0$, there is a constant $\delta>0$ such that
$K_{g}\leq -\delta.$
\end{lem}
\begin{proof}
Since $g$ is complete and $K_{g}$ has a lower bound, the background metric is enough to apply the Cheng-Yau maximum principle.
Since the right hand side of the equation has quadratic growth, from the Cheng-Yau maximum principle (Lemma \ref{Cheng-Yau}), $K_{g}$ has a upper bound and $c\sup K_g(\sup K_g+d)\leq 0$. Then $\sup K_g\leq 0$.  Then by the strong maximum principle, either $K_g<0$ or $K_g\equiv 0$.

Lift $(\Sigma,g)$ to the cover $(\widetilde \Sigma\cong \mathbb D,\widetilde g)$. Then the condition $g\leq Cg_{hyp}$ implies  $\widetilde g\leq Cg_{\mathbb D}$. Then together with Lemma \ref{KC}, we may apply Lemma \ref{Wan} to $$\triangle_{\widetilde g}(-K_{\widetilde g})\leq -cK_{\widetilde g}(K_{\widetilde g}+d)\leq -cdK_{\widetilde g}.$$  Then we obtain that $K_{\widetilde g}\leq -\delta$ for  some constant $\delta>0.$ So is $K_g.$
\end{proof}

\section{Preliminaries on Higgs bundles}\label{pre}
In this section, we recall some facts on Higgs bundles we use in this article. One may refer \cite{LiIntroduction} for some content of this section.
\subsection{Higgs bundles over Riemann surfaces}
Let $\Sigma$ be a Riemann surface and $K$ be the canonical line bundle of $\Sigma$.
\begin{df}
A Higgs bundle over a Riemann surface $\Sigma$ is a pair $(E,\phi)$ where $E$ is a holomorphic bundle over $\Sigma$ of rank $r$ and $\phi:E\rightarrow E\otimes K$ is holomorphic. Moreover a $SL(r,\mb{C})$-Higgs is a Higgs bundle with $\det E=\mathcal{O}$ and $\text{tr}\phi=0$.
\end{df}

Let $h$ be a Hermitian metric on $E$. We also denote $h$ as the induced Hermitian metric on $End(E)$.

\begin{df}
A Hermitian metric $h$ on a Higgs bundle is called harmonic, if it satisfies the Hitchin equation
\[F(h)+[\phi, \phi^{*_h}]=0,\]
where $F(h)$ is the curvature of the Chern connection $\nabla$ with respect to $h$, $\phi^{*_h}$ is the adjoint of $\phi$ with respect to the metric $h$, and the bracket $[,]$ is the Lie bracket on $End(E)$-valued $1$-forms.
\end{df}

Locally, suppose $\phi=fdz$ for $f$ a local holomorphic section of $End(E)$. We have the following useful estimate.
\begin{lem}\label{SimpsonEstimate}
For a holomorphic section $s$ of $End(E)$, then locally, we have
\begin{equation*}
\partial_z\partial_{\bar z}\log |s|_h^2\geq \frac{|[s,f^{*_h}]|_h^2-|[s,f]|_h^2}{|s|_h^2}.
\end{equation*}
\end{lem}
\begin{proof}
Denote $\partial_{z,h}$ as the $(1,0)$ part of the Chern connection on $End(E)$ with respect to $h$. Since $s$ is holomorphic, we have the equality
$\partial_z\partial_{\bar z}|s|_h^2=|\partial_{z,h}s|^2+h(s, \partial_{\bar z}\partial_{z,h}s).$ Then
\begin{eqnarray*}
\partial_z\partial_{\bar z}\log |s|_h^2=\frac{\partial_z\partial_{\bar z}|s|_h^2}{|s|_h^2}-\frac{\partial_z|s|_h^2}{|s|_h^2}\frac{\partial_{\bar z}|s|_h^2}{|s|_h^2}=\frac{h(s, \partial_{\bar z}\partial_{z,h}s)}{|s|_h^2}+\frac{|\partial_{z,h}s|^2}{|s|_h^2}-\frac{\partial_z|s|_h^2}{|s|_h^2}\frac{\partial_{\bar z}|s|_h^2}{|s|_h^2}.
\end{eqnarray*}
Notice that
\begin{eqnarray*}
\partial_z|s|_h^2\partial_{\bar z}|s|_h^2=|\partial_{\bar{z}}|s|_h^2|^2=|h(\partial_{z,h}s,s)|^2\leq |\partial_{z,h}s|_h^2|s|^2_h.
\end{eqnarray*}
So
\begin{eqnarray*}
\partial_z\partial_{\bar z}\log |s|_h^2\geq\frac{h(s, \partial_{\bar z}\partial_{z,h}s)}{|s|_h^2}=\frac{h(s, (\partial_{\bar z}\partial_{z,h}-\partial_{z,h}\partial_{\bar z})s)}{|s|_h^2}=\frac{h(s, -F^{End(E)}(h)s)}{|s|_h^2}.
\end{eqnarray*}
By the Hitchin equation, locally we have $F(h)+[\phi, \phi^{*_h}]=F(h)+[f,f^{*_h}]dz\wedge d\bar z=0$. So $$-F^{End(E)}(h)s=[[f,f^{*_h}],s].$$
So we have
\begin{eqnarray*}
\partial_z\partial_{\bar z}\log |s|_h^2\geq \frac{h(s, [[f, f^{*_h}],s])}{|s|_h^2}&=&\frac{h(s, [f,[ f^{*_h},s]])}{|s|_h^2}-\frac{h(s, [f^{*_h},[f,s]])}{|s|_h^2}.
\end{eqnarray*}
Since $h(u,[v,w])=h([v^{*_h},u],w)$, we obtain $\partial_z\partial_{\bar z}\log |s|_h^2\geq \frac{|[s,f^{*_h}]|_h^2-|[s,f]|_h^2}{|s|_h^2}.$
\end{proof}

The Hitchin equation together with the holomorphicity of $\phi$ gives a $\rho$-equivariant harmonic map $f:\widetilde{\Sigma}\rightarrow N:=SL(r,\mb{C})/SU(r)$, where $\widetilde{\Sigma}$ is the universal cover of $\Sigma$ and $\rho$ is the holonomy representation of the flat connection $D=\nabla+\phi+\phi^{*_h}$. The Hopf differential of $f$ is $\text{Hopf}(f)=2r\text{tr}(\phi^2)$ and the energy density of $f$ is $e(f)=2r\text{tr}(\phi\phi^*).$


\subsection{Toda system}

Let $\Sigma$ be a Riemann surface.  
Let $g=\tilde{g}dz\otimes d\bar z$ be a K\"ahler metric of $\Sigma$. The K\"ahler form $\omega$ is given by $\omega=\frac{\sqrt{-1}}{2}\tilde{g}dz\wedge d\bar z$ with respect to Re$(g)$. (Notice that in some literature the K\"ahler form $\omega$ is to $g+\bar{g}$.)

The cyclic Higgs bundles in the Hitchin component 
have the following form:
\[E=K^{\frac{r-1}{2}}\oplus K^{\frac{r-3}{2}}\oplus \cdots \oplus K^{\frac{3-r}{2}}\oplus K^{\frac{1-r}{2}}, \quad \phi=\begin{pmatrix}&&&&q\\1&&&&\\&1&&&\\&&\ddots&&\\&&&1&\end{pmatrix},\]
where $q$ is a holomorphic $r$-differential on $\Sigma.$

We also view the K\"ahler metric $g$ as a Hermitian metric on $K^{-1}$. 
Let $F(g)$ denote the curvature of the Chern connection of $K^{-1}$ with $g$. So $F(g)=-\partial\bar\partial\log \tilde{g}$.
The Laplacian with respect to $g$, the Gaussian curvature of $\Re(g)=\frac{1}{2}(g+\bar g)$, and the square norm of $q$ with respect to $g$ are respectively:
\[\triangle_g=\frac{1}{\tilde{g}}\partial_z\partial_{\bar z},\quad K_g=-\frac{2}{\tilde{g}}\partial_z\partial_{\bar z}\log \tilde{g},\quad |q|_g^2=\frac{q(z)\bar q(z)}{\tilde{g}^r}.\]

The Hermitian metric $g$ induces the Hermitian metrics $(g^{-1})^{\otimes (r+1-2i)/2}$ on $K^{(r+1-2i)/2}$, and a diagonal Hermitian metric $\bigoplus\limits_{i=1}^r(g^{-1})^{\otimes (r+1-2i)/2}$ on $E$. 
For any $\real^r$-valued function $\vecw=(w_1,\ldots,w_r)$, we define a Hermitian metric on $E$,
$$h(g,\vecw):=\bigoplus\limits_{i=1}^{r} e^{w_i}(g^{-1})^{\otimes (r+1-2i)/2}.$$
\begin{rem}\label{equiv}
Fix the Riemann surface $\Sigma$, consider another K\"ahler metric $g^{\prime}=e^fg$. Let $\vecw^{\prime}=(w_1+\frac{r-1}{2}f,\ldots,w_r+\frac{1-r}{2}f)$. Then $h(g,\vecw)=h(g^{\prime},\vecw^{\prime})$.
\end{rem}
By direct calculation, for $h(g,\vecw)$, it being harmonic is equivalent to it satisfying the Toda system. One may refer \cite{LiMochizuki}.
\begin{prop}
The Hermitian metric $h(g,\vecw)$ is harmonic if and only if
$\vecw$ satisfies the following Toda system:
\begin{eqnarray}\label{eq;20.9.21.1}
\triangle_g w_1&=&|q|_g^2e^{w_1-w_r}-e^{w_2-w_1}-\frac{r-1}{4}K_g, \nonumber\\
\triangle_g w_i&=&e^{w_i-w_{i-1}}-e^{w_{i+1}-w_{i}}-\frac{r+1-2i}{4}K_g,\quad i=2,\ldots,r-1\\
\triangle_g w_r&=&e^{w_r-w_{r-1}}-|q|_g^2e^{w_1-w_r}-\frac{1-r}{4}K_g. \nonumber
\end{eqnarray}
and $\sum\limits_{i=1}^r w_i=0.$
In particular, in the case $r=2$, the Hermitian metric $h(g,\vecw)$ if and only if $\vecw=(w_1,-w_1)$ satisfies
\begin{eqnarray}\label{harmonic}
\triangle_g w_1=|q|_g^2e^{2w_1}-e^{-2w_1}+\frac{1}{4}.
\end{eqnarray}
\end{prop}

\begin{df}
A solution $(w_1,\cdots, w_r)$ to the Toda system (\ref{eq;20.9.21.1}) is said to be real if $w_i+w_{r+1-i}=0$ for $i=1,\cdots, r$.
\end{df}

Denote $n=[\frac{r}{2}]$. Then the real solution to the Toda system is given by

\begin{eqnarray}\label{todathmPreliminary}
\triangle_{g}w_1&=&e^{2w_1}|q|^2_{g}-e^{-w_1+w_2}-\frac{r-1}{4}K_g,\nonumber\\
\triangle_{g}w_i&=&e^{-w_{i-1}+w_i}-e^{-w_i+w_{i+1}}-\frac{r+1-2i}{4}K_g,\quad 2\leq i\leq n-1,\\
\triangle_{g}w_n&=&e^{-w_{n-1}+w_n}-e^{-(2n+2-r)w_n}-\frac{r+1-2n}{4}K_g\nonumber.
\end{eqnarray}

We also consider the subcyclic Higgs bundles in the Hitchin component, 
which have the following form:
\[E=K^{\frac{r-1}{2}}\oplus K^{\frac{r-3}{2}}\oplus \cdots \oplus K^{\frac{3-r}{2}}\oplus K^{\frac{1-r}{2}}, \quad \phi=\begin{pmatrix}&&&q&0\\1&&&&q\\&1&&&\\&&\ddots&&\\&&&1&\end{pmatrix},\]
where $q$ is a holomorphic $(r-1)$-differential on $\Sigma.$ Similar to the cyclic case, the diagonal harmonic metric $h(g,\vecw)$ gives a variant Toda system as follows:
\begin{eqnarray}\label{subtoda}
\triangle_{g}w_1&=&e^{w_1+w_2}|q|^2_{g}-e^{-w_1+w_2}+\frac{r-1}{4},\nonumber\\
\triangle_{g}w_2&=&e^{w_1+w_2}|q|^2_{g}+e^{-w_1+w_2}-e^{-w_2+w_3}+\frac{r-3}{4},\nonumber\\
\triangle_{g}w_i&=&e^{-w_{i-1}+w_i}-e^{-w_i+w_{i+1}}+\frac{r+1-2i}{4},\quad 3\leq i\leq r-2,\\
\triangle_{g}w_{r-1}&=&e^{-w_{r-2}+w_{r-1}}-e^{w_1+w_2}|q|^2_{g}-e^{-w_{r-1}+w_r}+\frac{3-r}{4},\nonumber\\
\triangle_{g}w_r&=&e^{-w_{r-1}+w_r}-e^{w_1+w_2}|q|^2_{g}+\frac{1-r}{4}.\nonumber
\end{eqnarray}
If the solution is real, i.e. $w_i+w_{r+1-i}=0$ for $i=1,\cdots, r$, then the variant Toda system becomes
\begin{eqnarray}\label{vtoda}
\triangle_{g}w_1&=&e^{w_1+w_2}|q|^2_{g}-e^{-w_1+w_2}-\frac{r-1}{4}K_g,\nonumber\\
\triangle_{g}w_2&=&e^{w_1+w_2}|q|^2_{g}+e^{-w_1+w_2}-e^{-w_2+w_3}-\frac{r-3}{4}K_g,\nonumber\\
\triangle_{g}w_i&=&e^{-w_{i-1}+w_i}-e^{-w_i+w_{i+1}}-\frac{r+1-2i}{4}K_g,\quad 3\leq i\leq n-1,\\
\triangle_{g}w_n&=&e^{-w_{n-1}+w_n}-e^{-(2n+2-r)w_n}-\frac{r+1-2n}{4}K_g\nonumber.
\end{eqnarray}
where $n=[\frac{r}{2}]$.

For both the system (\ref{eq;20.9.21.1}) and (\ref{subtoda}), denote $$g(h)_{i}=e^{-w_{i}+w_{i+1}}g, \quad i=1,\cdots,r-1.$$ We consider the complete solution to the system (\ref{eq;20.9.21.1}) or (\ref{subtoda}) in the following sense.
\begin{df}
A solution $(w_1,\cdots, w_r)$ to the system (\ref{eq;20.9.21.1}) or (\ref{subtoda}) is complete if the metrics $g(h)_{i}$, $i=1,\cdots,r-1$ are complete.
\end{df}
In terms of harmonic metrics, it is equivalent to the condition that on $\Sigma$ the K\"ahler metrics $h_{K^{(r+1-2(i+1))/2}}\otimes h_{K^{(r+1-2i)/2}}^{-1}$ are complete, which is independent of the background K\"ahler metric $g$.

The complete solution uniquely exists and satisfies some estimates. For the Toda system, it is from \cite{LiMochizuki}.
\begin{thm}(\cite[Theorem 1.3 and 1.6]{LiMochizuki})\label{LM1}
Let $r\geq 2$. Let $q$ be a holomorphic $r$-differential on $\Sigma$. Let $g$ be a K\"ahler metric on $\Sigma$.
Then there is a unique complete solution $(w_1,\cdots, w_r)\in C^{\infty}(\Sigma)$ to the Toda system (\ref{eq;20.9.21.1}). Moreover, it is real.
Furthermore, $w_i$'s satisfy
$$\frac{(i-1)(r-i+1)}{i(r-i)}\leq\frac{e^{-w_{i-1}+w_{i}}}{e^{-w_{i}+w_{i+1}}}<1,~i=1,\cdots,n,$$
unless $\Sigma$ is parabolic and $q$ has no zeros.
Here $w_0=-w_1-\log|q|_{g}^2$, $w_{n+1}=-(2n+1-r)w_n$.
\end{thm}
For the variant Toda system, it is from \cite{Sagman}.
\begin{thm}(\cite{Sagman})\label{varianttoda1}
Let $r\geq 3$. Let $q$ be a holomorphic $(r-1)$-differential on $\Sigma$. Let $g$ be a K\"ahler metric on $\Sigma$.
Then there is a unique complete solution $(w_1,\cdots, w_r)\in C^{\infty}(\Sigma)$ to the variant Toda system (\ref{subtoda}). Moreover, it is real.
Furthermore, $w_i$'s satisfy
$$\frac{e^{-w_{i-1}+w_{i}}}{e^{-w_{i}+w_{i+1}}}<1,~i=1,\cdots,n,\text{ and }e^{-w_0+w_1}+e^{-w_1+w_2}<e^{-w_2+w_3},$$
unless $\Sigma$ is parabolic and $q$ has no zeros.
Here $w_0=-w_2-\log|q|_{g}^2$, $w_{n+1}=-(2n+1-r)w_n$.
\end{thm}


\section{Bounded differentials and harmonic maps 
}\label{General_r}
\subsection{Single equation}
We first consider the case of a single equation. For the geometric applications, we consider a broader class of equations than (\ref{harmonic}). Let $(\Sigma,g)$ be a Riemann surface with a K\"ahler metric $g$. Consider
\begin{eqnarray}\label{ab}
\triangle_{g}w=-\kappa(e^{aw}|q|^2_{g}-e^{-bw})-cK_g,
\end{eqnarray}
where $a,b,c$ are positive constants, $\kappa$ is a function on $\Sigma$ satisfying $-C_1\leq \kappa\leq -C_2$ for some constant $C_1\geq C_2>0$.

\begin{lem}\label{n=1}
Let $w$ be a solution to Equation (\ref{ab}). Suppose $e^{-bw}g$ is complete. Then $|q|^2_{g}e^{(a+b)w}<1$.
Furthermore if $bc=\frac{1}{2}$, which is the situation for $r=2,3$ in the Toda system (\ref{todathmPreliminary}) and for $r=3$ in the variant Toda system (\ref{vtoda}), then $K_{e^{-bw}g}<0$.
\end{lem}
\begin{proof}
Denote $g=\ti{g}(z)dz\otimes d\bar{z}$. Notice that for a Riemannian metric $g$ and a smooth function $f$,
\begin{eqnarray}\label{1}
\triangle_g e^f=e^f(\triangle_g f+|\nabla_g f|^2_g)\geq e^f\triangle_g f.
\end{eqnarray}
Then outside the zeros of $q$, we have
\begin{eqnarray}\label{10}
\triangle_g (|q|^2_g e^f)=\triangle_g e^{f+\log |q|^2_g}\geq |q|^2_g e^f(\triangle_g f+\triangle_g \log |q|^2_g).
\end{eqnarray}
Recall locally, $\triangle_g=\frac{1}{g}\partial_{z}\partial_{\bar{z}}$, $|q|^2_{g}=\frac{q\bar{q}}{g^r}$. Since $q$ is holomorphic, we have $$\triangle_{g} \log |q|^2_{g}=-r\triangle_{g} \log \tilde{g}=\frac{r}{2}K_g.$$
Then
\begin{eqnarray}\label{2}
\triangle_{g} (|q|^2_{g} e^f)\geq |q|^2_{g} e^f(\triangle_{g} f+\frac{r}{2}K_g).
\end{eqnarray}
outside the zeros of $q$. By the continuity, it holds everywhere. So for $u=|q|^2_{g} e^{(a+b)w}-1$,
\begin{eqnarray}\label{11}
\triangle_{e^{-bw}g}u&=&\triangle_{e^{-bw}g} (|q|^2_{g} e^{(a+b)w})\\
&\geq& |q|^2_{g}e^{(a+b)w}((a+b)\triangle_{e^{-bw}g} w+(a+b)cK_g)\nonumber\\
&=&|q|^2_{g} e^{(a+b)w}((a+b)(-\kappa(|q|^2_{g}e^{(a+b)w}-1)-cK_g))\nonumber\\
&=&-(a+b)\kappa|q|^2_{g} e^{(a+b)w}(|q|^2_{g}e^{(a+b)w}-1)\nonumber\\
&=&-(a+b)\kappa u(u+1).\nonumber
\end{eqnarray}

The Gaussian curvature of the metric $e^{-bw}g$ is
\begin{eqnarray}
K_{e^{-bw}g}&=&\frac{-2}{e^{-bw}\ti{g}}\partial_z\partial_{\bar{z}}\log (e^{-bw}\ti{g})\\
&=& \frac{-2}{e^{-bw}}(\triangle_{g}(-bw)-\frac{K_g}{2})\nonumber\\
&=& \frac{-2}{e^{-bw}}\big(bK(e^{aw}|q|^2_{g}-e^{-bw})+bcK_g-\frac{K_g}{2}\big)\nonumber\\
&\geq& 2b\kappa+(bc-\frac{1}{2})K_g\nonumber\\
&\geq & C\nonumber,
\end{eqnarray}
for some constant $C$.
Since $e^{-bw}g$ is complete and has curvature bounded from below, then $|q|^2_{g}e^{(a+b)w}$ has an upper bound from the Cheng-Yau maximum principle (Lemma \ref{Cheng-Yau}).

Then by the strong maximum principle, either $|q|^2_{g}e^{(a+b)w}<1$ or $|q|^2_{g}e^{(a+b)w}\equiv 1$. If $|q|^2_{g}e^{(a+b)w}\equiv 1$, then the equality holds in (\ref{11}), which implies the equality in (\ref{10}) and (\ref{1}). So $w$ is a constant. Then $c=0$ from Equation (\ref{ab}), contradiction. So $|q|^2_{g}e^{(a+b)w}<1$.

If $bc=\frac{1}{2}$, then
\begin{eqnarray}\label{-bw}
K_{e^{-bw}g}
&=& \frac{-2}{e^{-bw}}\big(bK(e^{aw}|q|^2_{g}-e^{-bw})+bcK_g-\frac{K_g}{2}\big)\nonumber\\
&=& -2b\kappa(e^{(a+b)w}|q|^2_{g}-1)\nonumber\\
&<& 0.
\end{eqnarray}
\end{proof}
Wan \cite{WAN} showed some results on the boundedness for the equation (\ref{harmonic}), i.e. the case $r=2$. For the geometric applications, we consider the broader class of equations (\ref{ab}). The proof is similar. For the convenience of the readers, we include the proof. We will discuss the geometric applications in Section \ref{r=2} and Section \ref{r=3}. From Remark \ref{equiv}, we may assume $(\Sigma,g)=(\mb{D},g_{\mb{D}})$ or $(\mb{C},g_{\mb{C}})$. Since on $\mb{C}$, bounded holomorphic differentials are less interesting, for simplicity, we assume $(\Sigma,g)=(\mb{D},g_{\mb{D}})$.
\begin{prop}\label{n=1 bd}
Let $w$ be a solution to the equation (\ref{ab}) on $(\mb{D},g_{\mb{D}})$. We further assume $bc=\frac{1}{2}$, which is the situation for $r=2,3$ in the Toda system (\ref{todathmPreliminary}) and for $r=3$ in the variant Toda system (\ref{vtoda}). Suppose $e^{-bw}g_{\mb{D}}$ is complete. Then the followings are equivalent:

(1) $|q|_{g_{\mb{D}}}$ is bounded.

(2) $|w|$ is bounded.

(3) There is a constant $C>0$ such that $|q|^2_{g_{\mb{D}}}e^{aw}+e^{-bw}\leq C$.

(4) There is a constant $\delta>0$ such that the curvature of the metric $e^{-bw}g_{\mb{D}}$ satisfying $K_{e^{-bw}g_{\mb{D}}}\leq -\delta.$

(5) There is a constant $\delta>0$ such that $|q|^2_{g_{\mb{D}}}e^{(a+b)w}\leq 1-\delta$.
\end{prop}
\begin{proof}
First we notice that $w$ has an upper bound automatically. In fact,
\begin{\eq}
\triangle_{e^{-bw}g_{\mb{D}}}w&=&\frac{1}{e^{-bw}}\triangle_{g_{\mb{D}}}w=-K(e^{(a+b)w}|q|^2_{g_{\mb{D}}}-1)+ce^{bw}\geq ce^{bw}-C_1,
\end{\eq}
since $e^{-bw}g_{\mb{D}}$ is complete, from the Cheng-Yau maximum principle Lemma \ref{Cheng-Yau}, $w$ has an upper bound.

(1)$\Rightarrow$(2): We only need to show $w$ has a lower bound. In fact, suppose $|q|^2_{g_{\mb{D}}}\leq C$, then
$$\triangle_{g_{\mb{D}}}(-w)=K(e^{-a(-w)}|q|^2_{g_{\mb{D}}}-e^{b(-w)})-c\geq C_2e^{b(-w)}-C_1Ce^{-a(-w)}-c.$$
Then from the Cheng-Yau maximum principle Lemma \ref{Cheng-Yau}, $-w$ has an upper bound. So $w$ has a lower bound.

(2)$\Rightarrow$(4):
From Lemma \ref{n=1}, $K_{e^{-bw}g_{\mb{D}}}<0$.
Let $u=1-|q|^2_{g_{\mb{D}}}e^{(a+b)w}$, from $(\ref{11})$,
\begin{\eq}
\triangle_{e^{-bw}g_{\mb{D}}}u=\frac{-1}{e^{-bw}}\triangle_{g_{\mb{D}}}e^{(a+b)w}|q|^2_{g_{\mb{D}}}\leq K(a+b)|q|^2_{g_{\mb{D}}} e^{(a+b)w}(|q|^2_{g_{\mb{D}}}e^{(a+b)w}-1)=-K(a+b)|q|^2_{g_{\mb{D}}} e^{(a+b)w}u.
\end{\eq}
By condition (2), there exists a constant $C<0$ such that $\triangle_{e^{-bw}g_{\mb{D}}}u\leq Cu$. Since $e^{-bw}g_{\mb{D}}$ is complete and negative curved, we obtain $K_{e^{-bw}g_{\mb{D}}}\leq -\delta$ for some $\delta>0$ from Lemma \ref{Wan}.

(4)$\Leftrightarrow$(5): It is obvious from formula (\ref{-bw}).

(5)$\Rightarrow$(2): We only need to show $w$ has a lower bound. In fact,
$$\triangle_{g_{\mb{D}}}(-w)=-K(1-e^{(a+b)w}|q|^2_{g_{\mb{D}}})e^{-bw}-c\geq C_2\delta e^{b(-w)}-c.$$
Then from the Cheng-Yau maximum principle, $w$ has a lower bound.

(2)$\Rightarrow$(3): It follows from Lemma \ref{n=1}.

(3)$\Rightarrow$(1): From the assumption, we have $|q|^2_{g_{\mb{D}}}e^{aw}\leq C$, $e^{-bw}\leq C$. So $w$ has a lower bound. Then $|q|^2_{g_{\mb{D}}}$ is bounded. We finish the proof.
\end{proof}

\subsection{Equation system}\label{EquationSystem}
Now consider the Toda system (\ref{todathmPreliminary}) and the variant Toda system (\ref{vtoda}). Since for lower $r$, the formulae cannot be written in a general form, we calculate the formulae case by case.

For the Toda system (\ref{todathmPreliminary}), let $f_0=e^{2w_1}|q|_g^2$, $f_i=e^{-w_i+w_{i+1}}$, $i=0, \cdots, r$, where $w_0=-w_1-\log |q|_g^2=-w_{r+1}$. Notice that $\triangle_g\log|q|_g^2=\frac{r}{2}K_g.$
\\
For $r=2$, the Toda system (\ref{todathmPreliminary}) implies
\begin{eqnarray*}
\triangle_g\log f_0&=&2f_0-2f_{1}+\frac{1}{2}K_g,\text{ outsides the zeros of }q,\\
\triangle_g\log f_1&=&2f_1-2f_0+\frac{1}{2}K_g.
\end{eqnarray*}
For $r=3$, the Toda system (\ref{todathmPreliminary}) implies
\begin{eqnarray*}
\triangle_g\log f_0&=&2f_0-2f_{1}+\frac{1}{2}K_g,\text{ outsides the zeros of }q,\\
\triangle_g\log f_1&=&f_1-f_0+\frac{1}{2}K_g.
\end{eqnarray*}
For $r\geq 4,$ the Toda system (\ref{todathmPreliminary}) implies
\begin{eqnarray*}
\triangle_g\log f_0&=&2f_0-2f_{1}+\frac{1}{2}K_g,\text{ outsides the zeros of }q,\\
\triangle_g\log f_i&=&2f_i-f_{i-1}-f_{i+1}+\frac{1}{2}K_g,\quad i=1,\cdots, r-1.
\end{eqnarray*}

For the variant Toda system (\ref{vtoda}), let $f_0=e^{w_1+w_2}|q|^2$, $f_i=e^{-w_i+w_{i+1}}$, $i=0, \cdots, r$, where $w_0=-w_2-\log |q|_g^2=-w_{r+1}$.
\\
For $r=3$, the variant Toda system (\ref{vtoda}) implies
\begin{eqnarray*}
\triangle_g\log f_0&=&f_0-f_1+\frac{1}{2}K_g,\text{ outsides the zeros of }q\\
\triangle_g\log f_1&=&f_1-f_0+\frac{1}{2}K_g.
\end{eqnarray*}
For $r=4$, the variant Toda system (\ref{vtoda}) implies
\begin{eqnarray*}
\triangle_g\log f_0&=&2f_0-f_2+\frac{1}{2}K_g,\text{ outsides the zeros of }q\\
\triangle_g\log f_1&=&2f_1-f_2+\frac{1}{2}K_g,\\
\triangle_g\log f_2&=&2f_2-2f_0-2f_1+\frac{1}{2}K_g.
\end{eqnarray*}
For $r=5$, the variant Toda system (\ref{vtoda}) implies
\begin{eqnarray*}
\triangle_g\log f_0&=&2f_0-f_2+\frac{1}{2}K_g,\text{ outsides the zeros of }q\\
\triangle_g\log f_1&=&2f_1-f_2+\frac{1}{2}K_g,\\
\triangle_g\log f_2&=&f_2-f_0-f_1+\frac{1}{2}K_g.
\end{eqnarray*}
For $r\geq 6$, the variant Toda system (\ref{vtoda}) implies
\begin{eqnarray*}
\triangle_g\log f_0&=&2f_0-f_2+\frac{1}{2}K_g,\text{ outsides the zeros of }q\\
\triangle_g\log f_1&=&2f_1-f_2+\frac{1}{2}K_g,\\
\triangle_g\log f_2&=&2f_2-f_0-f_1-f_3+\frac{1}{2}K_g,\\
\triangle_g\log f_i&=&2f_i-f_{i-1}-f_{i+1}+\frac{1}{2}K_g,\quad i=3, \cdots, r-1.
\end{eqnarray*}
Notice that for both (\ref{todathmPreliminary}) and (\ref{vtoda}), $f_{i}=e^{-w_{i}+w_{i+1}}=e^{w_{r+1-i}-w_{r-i}}=f_{r-i}$, for $0\leq i \leq r$.

\begin{lem}\label{Kn-1}
For the Toda system (\ref{todathmPreliminary}) and the variant Toda system (\ref{vtoda}), let $(w_1,\cdots, w_r)$ be a real solution. Denote $g(h)_{i}=e^{-w_{i}+w_{i+1}}g$, $i=1,\cdots,r-1,$ then the curvature $K_{g(h)_i}$ are bounded from below.

Moreover, there exist constant $c_1,c_2>0$ (maybe different in each case) such that
\begin{eqnarray*}
\triangle_{g(h)_{n}}K_{g(h)_{n}}&\geq&c_1K_{g(h)_{n}}(K_{g(h)_{n}}+c_2)\quad \text{for $r=2,3$ in (\ref{todathmPreliminary}), $r=3,4,5$ in (\ref{vtoda}).}\\
\triangle_{g(h)_{n-1}}K_{g(h)_{n-1}}&\geq&c_1K_{g(h)_{n-1}}(K_{g(h)_{n-1}}+c_2)\quad \text{for $r=4$ in (\ref{todathmPreliminary}), $r=7$ in (\ref{vtoda}).}
\end{eqnarray*}
Moreover, suppose $(w_1,\cdots, w_r)$ is a complete solution. Then
\begin{eqnarray*}
\triangle_{g(h)_{n}}K_{g(h)_{n}}&\geq&c_1K_{g(h)_{n}}(K_{g(h)_{n}}+c_2) \quad \text{for $r\geq 2$ in (\ref{todathmPreliminary}), $r\geq 3$ in (\ref{vtoda})}.\\
\triangle_{g(h)_{n-1}}K_{g(h)_{n-1}}&\geq&c_1K_{g(h)_{n-1}}(K_{g(h)_{n-1}}+c_2) \quad \text{for $r\geq4$, $r\neq 5$ in (\ref{todathmPreliminary}), $r\geq 6$ in (\ref{vtoda})}.\label{CurvatureEquation2}
\end{eqnarray*}
\end{lem}
\begin{proof}
Locally, denote $g=\tilde{g}(z)dz\otimes d\bar{z}$. We first consider the Toda system (\ref{todathmPreliminary}).
\\
For $r=2$,
\begin{eqnarray*}
K_{g(h)_1}=-\frac{2}{f_1\tilde{g}}\partial_z\partial_{\bar{z}}\log(f_1\tilde{g})=4(\frac{f_0}{f_1}-1).
\end{eqnarray*}
Then
\begin{eqnarray*}
\triangle_{g(h)_1}K_{g(h)_1}
&=&4\triangle_{g(h)_1}\frac{f_{0}}{f_1}\\
&\geq& 4\frac{f_{0}}{f_1}\triangle_{g(h)_1}\log\frac{f_{0}}{f_1}\text{\quad outsides the zeros of $q$}\\
&=&16\frac{f_{0}}{f_{1}^2}(f_0-f_1)\\
&=&K_{g(h)_1}(K_{g(h)_1}+4).
\end{eqnarray*}
Since $q$ is holomorphic, its zeros are discrete, then from the continuity the inequality holds everywhere.
\\
For $r=3$, similarly
\begin{eqnarray*}
&&K_{g(h)_1}=-\frac{2}{f_1\tilde{g}}\partial_z\partial_{\bar{z}}\log(f_1\tilde{g})=2(\frac{f_0}{f_1}-1).\\
&&\triangle_{g(h)_1}K_{g(h)_1}
=2\triangle_{g(h)_1}\frac{f_{0}}{f_1}
\geq 2\frac{f_{0}}{f_1}\triangle_{g(h)_1}\log\frac{f_{0}}{f_1}
=6\frac{f_{0}}{f_{1}^2}(f_0-f_1)
=\frac{3}{2}K_{g(h)_1}(K_{g(h)_1}+2).
\end{eqnarray*}
For $r=4$, notice that $f_3=f_1$,
\begin{eqnarray*}
K_{g(h)_1}&=&-\frac{2}{f_1\tilde{g}}\partial_z\partial_{\bar{z}}\log(f_1\tilde{g})=2(\frac{f_0+f_2}{f_1}-2).\\
\triangle_{g(h)_1}K_{g(h)_{1}}&=&2\triangle_{g(h)_1}(\frac{f_{0}}{f_1}+\frac{f_{2}}{f_1})\\
&\geq& 2(\frac{f_{0}}{f_1}\triangle_{g(h)_1}\log\frac{f_{0}}{f_1}+\frac{f_{2}}{f_1}\triangle_{g(h)_1} \log \frac{f_{2}}{f_1})\\
&=&2(\frac{f_{0}}{f_{1}^2}(3f_0-4f_1+f_2)+\frac{f_{2}}{f_1^2}(f_0-3f_1+3f_2-f_1))\\
&=&2(3f_0^2-4f_0f_1+f_0f_2-4f_1f_2+3f_2^2+f_0f_2)/f_1^2\\
&\geq&2(2(f_0+f_2)^2-4(f_0+f_2)f_1)/f_1^2\\
&=&K_{g(h)_1}(K_{g(h)_1}+4).
\end{eqnarray*}

Now we consider the variant Toda system (\ref{vtoda}) similarly.\\
For $r=3$,
\begin{eqnarray*}
&&K_{g(h)_1}=-\frac{2}{f_1\tilde{g}}\partial_z\partial_{\bar{z}}\log(f_1\tilde{g})=2(\frac{f_0}{f_1}-1).\\
&&\triangle_{g(h)_1}K_{g(h)_1}
=2\triangle_{g(h)_1}\frac{f_{0}}{f_1}
\geq 2\frac{f_{0}}{f_1}\triangle_{g(h)_1}\log\frac{f_{0}}{f_1}
=4\frac{f_{0}}{f_{1}^2}(f_0-f_1)
=K_{g(h)_1}(K_{g(h)_1}+2).
\end{eqnarray*}
For $r=4$,
\begin{eqnarray*}
K_{g(h)_2}&=&-\frac{2}{f_2\tilde{g}}\partial_z\partial_{\bar{z}}\log(f_2\tilde{g})=4\big(\frac{f_0+f_1}{f_2}-1\big).\\
\triangle_{g(h)_2}K_{g(h)_2}
&=&4\triangle_{g(h)_2}(\frac{f_{0}}{f_2}+\frac{f_{1}}{f_2})\\
&\geq& 4(\frac{f_{0}}{f_2}\triangle_{g(h)_2}\log\frac{f_{0}}{f_2}+\frac{f_{1}}{f_2}\triangle_{g(h)_2} \log \frac{f_{1}}{f_2})\\
&=& 4(\frac{f_{0}}{f_{2}^2}(4f_0+2f_1-3f_2)+\frac{f_{1}}{f_2^2}(4f_1-3f_2+2f_0))\\
&=&4(4f_0^2+2f_0f_1-3f_0f_2+4f_1^2-3f_1f_2+2f_0f_1)/f_2^2\\
&\geq&4(3(f_0+f_1)^2-3(f_0+f_1)f_2)/f_2^2\\
&=&\frac{3}{4}K_{g(h)_2}(K_{g(h)_2}+4).
\end{eqnarray*}
And
\begin{eqnarray*}
&&K_{g(h)_1}=-\frac{2}{f_1\tilde{g}}\partial_z\partial_{\bar{z}}\log(f_1\tilde{g})=2(\frac{f_2}{f_1}-2).\\
&&\triangle_{g(h)_1}K_{g(h)_1}
=2\triangle_{g(h)_1}\frac{f_{2}}{f_1}
\geq 2\frac{f_{2}}{f_1}\triangle_{g(h)_1}\log\frac{f_{2}}{f_1}
=2\frac{f_{2}}{f_{1}^2}(-2f_0-4f_1+3f_2).
\end{eqnarray*}
Suppose the solution is complete, then from Theorem \ref{varianttoda1}, $f_2>f_1+f_0.$ So
\begin{\eq}
\triangle_{g(h)_1}K_{g(h)_1}\geq 2\frac{f_{2}}{f_{1}^2}(-2f_1+f_2)=\frac{1}{2}K_{g(h)_1}(K_{g(h)_1}+4).
\end{\eq}
For $r=5$,
\begin{eqnarray*}
K_{g(h)_2}&=&-\frac{2}{f_2\tilde{g}}\partial_z\partial_{\bar{z}}\log(f_2\tilde{g})=2\big(\frac{f_0+f_1}{f_2}-1\big).\\
\triangle_{g(h)_2}K_{g(h)_2}
&=&2\triangle_{g(h)_2}(\frac{f_{0}}{f_2}+\frac{f_{1}}{f_2})\\
&\geq& 2(\frac{f_{0}}{f_2}\triangle_{g(h)_2}\log\frac{f_{0}}{f_2}+\frac{f_{1}}{f_2}\triangle_{g(h)_2} \log \frac{f_{1}}{f_2})\\
&=& 2(\frac{f_{0}}{f_{2}^2}(3f_0+f_1-2f_2)+\frac{f_{1}}{f_2^2}(3f_1-2f_2+f_0))\\
&=&2(3f_0^2+f_0f_1-2f_0f_2+3f_1^2-2f_1f_2+f_0f_1)/f_2^2\\
&\geq&2(2(f_0+f_1)^2-2(f_0+f_1)f_2)/f_2^2\\
&=&K_{g(h)_2}(K_{g(h)_2}+2).
\end{eqnarray*}
And
\begin{eqnarray*}
&&K_{g(h)_1}=-\frac{2}{f_1\tilde{g}}\partial_z\partial_{\bar{z}}\log(f_1\tilde{g})=2(\frac{f_2}{f_1}-2).\\
&&\triangle_{g(h)_1}K_{g(h)_1}
=2\triangle_{g(h)_1}\frac{f_{2}}{f_1}
\geq 2\frac{f_{2}}{f_1}\triangle_{g(h)_1}\log\frac{f_{2}}{f_1}
=2\frac{f_{2}}{f_{1}^2}(-f_0-3f_1+2f_2).
\end{eqnarray*}
Suppose the solution is complete, then from Theorem \ref{varianttoda1}, $f_2>f_1+f_0.$ So
\begin{\eq}
\triangle_{g(h)_1}K_{g(h)_1}\geq 2\frac{f_{2}}{f_{1}^2}(-2f_1+f_2)=\frac{1}{2}K_{g(h)_1}(K_{g(h)_1}+4).
\end{\eq}
For $r=6$, $r=7$,
\begin{eqnarray*}
K_{g(h)_2}&=&-\frac{2}{f_2\tilde{g}}\partial_z\partial_{\bar{z}}\log(f_2\tilde{g})=2\big(\frac{f_0+f_1+f_3}{f_2}-2\big).\\
\triangle_{g(h)_2}K_{g(h)_2}
&=&2\triangle_{g(h)_2}(\frac{f_{0}}{f_2}+\frac{f_{1}}{f_2}+\frac{f_3}{f_2})\\
&\geq& 2(\frac{f_{0}}{f_2}\triangle_{g(h)_2}\log\frac{f_{0}}{f_2}+\frac{f_{1}}{f_2}\triangle_{g(h)_2} \log \frac{f_{1}}{f_2}+\frac{f_{3}}{f_2}\triangle_{g(h)_2} \log \frac{f_{3}}{f_2})\\
&=& 2(\frac{f_{0}}{f_{2}^2}(3f_0+f_1+f_3-3f_2)+\frac{f_{1}}{f_2^2}(3f_1-3f_2+f_0+f_3)+\frac{f_3}{f_2^2}(f_0+f_1-3f_2+3f_3-f_4))\\
&=&2(3f_0^2+3f_1^2+2f_0f_1+2f_0f_3+2f_1f_3-3f_0f_2-3f_1f_2+3f_3^2-3f_2f_3-f_3f_4)/f_2^2\\
&=&2(2f_0^2+2f_1^2+(f_0+f_1)^2+2(f_0+f_1)f_3-3(f_0+f_1+f_3)f_2+3f_3^2-f_3f_4)/f_2^2\\
&=&2(2f_0^2+2f_1^2+2f_3^2+(f_0+f_1+f_3)^2-3(f_0+f_1+f_3)f_2-f_3f_4)/f_2^2\\
&\geq&2(\frac{3}{2}(f_0+f_1+f_3)^2-3(f_0+f_1+f_3)f_2+f_3^2-f_3f_4)/f_2^2\\
&=&\frac{3}{4}K_{g(h)_2}(K_{g(h)_2}+4)+2\frac{f_3^2-f_3f_4}{f_2^2}
\end{eqnarray*}
For $r=6$, we have $f_4=f_2$. Suppose the solution is complete, then from Theorem \ref{LM1}, $f_2<f_3$. So $$\triangle_{g(h)_2}K_{g(h)_2}\geq \frac{3}{4}K_{g(h)_2}(K_{g(h)_2}+4).$$
For $r=7$, we have $f_4=f_3$. So $$\triangle_{g(h)_2}K_{g(h)_2}\geq \frac{3}{4}K_{g(h)_2}(K_{g(h)_2}+4).$$
\\
For $r=6,7,8,9$, suppose the solution is complete, then from Theorem \ref{varianttoda1}, $f_0+f_1<f_2$, $f_4\geq f_5$, then
\begin{eqnarray*}
K_{g(h)_3}&=&-\frac{2}{f_3\tilde{g}}\partial_z\partial_{\bar{z}}\log(f_3\tilde{g})=2\big(\frac{f_2+f_4}{f_3}-2\big).\\
\triangle_{g(h)_3}K_{g(h)_3}
&=&2\triangle_{g(h)_3}(\frac{f_{2}}{f_3}+\frac{f_{4}}{f_3})\\
&\geq& 2(\frac{f_{2}}{f_3}\triangle_{g(h)_3}\log\frac{f_{2}}{f_3}+\frac{f_{4}}{f_3}\triangle_{g(h)_3} \log \frac{f_{4}}{f_3})\\
&=& 2(\frac{f_{2}}{f_{3}^2}(-f_0-f_1+3f_{2}-3f_{3}+f_{4})-\frac{f_{4}}{f_3^2}(-f_{2}+3f_3-3f_{4}+f_{5}))\\
&\geq&2(2f_{2}^2+3f_{4}^2+2f_{2}f_{4}-3f_{2}f_3-3f_{4}f_3-f_{4}f_{5})/f_3^2\\
&=&2(f_{2}+f_{4}-3f_3+\frac{f_{2}^2+f_{4}^2}{f_{2}+f_{4}}+\frac{f_4^2-f_4f_5}{f_{2}+f_{4}})\cdot \frac{f_{2}+f_{4}}{f_3^2}\\
&\geq&2(\frac{3}{2}(f_{2}+f_{4})-3f_3)\cdot \frac{f_{2}+f_{4}}{f_3^2}+\frac{f_4^2-f_4f_5}{f_3^2}\\
&\geq&\frac{3}{4}K_{g(h)_3}(K_{g(h)_3}+4).
\end{eqnarray*}

Last we consider $r\geq 4$, $2\leq i\leq r-2$ for (\ref{todathmPreliminary}) and $r\geq 8$, $4\leq i\leq r-2$ for (\ref{vtoda}) simultaneously.
\begin{eqnarray*}
K_{g(h)_i}=-\frac{2}{f_i\tilde{g}}\partial_z\partial_{\bar{z}}\log(f_i\tilde{g})=-\frac{2}{f_i}\triangle_{g}\log f_i+\frac{1}{f_i}K_g=2\big(\frac{f_{i-1}+f_{i+1}}{f_i}-2\big).
\end{eqnarray*}
From inequality (\ref{1}), we have
\begin{eqnarray*}
\triangle_{g(h)_i}K_{g(h)_i}
&=&2\triangle_{g(h)_i}(\frac{f_{i-1}}{f_i}+\frac{f_{i+1}}{f_i})\\
&\geq& 2(\frac{f_{i-1}}{f_i}\triangle_{g(h)_i}\log\frac{f_{i-1}}{f_i}+\frac{f_{i+1}}{f_i}\triangle_{g(h)_i} \log \frac{f_{i+1}}{f_i})\\
&=& 2(\frac{f_{i-1}}{f_{i}^2}(-f_{i-2}+3f_{i-1}-3f_{i}+f_{i+1})-\frac{f_{i+1}}{f_i^2}(-f_{i-1}+3f_i-3f_{i+1}+f_{i+2}))\\
&=&2(3f_{i-1}^2+3f_{i+1}^2+2f_{i-1}f_{i+1}-3f_{i-1}f_i-3f_{i+1}f_i-f_{i-2}f_{i-1}-f_{i+1}f_{i+2})/f_i^2\\
&=&2(f_{i-1}+f_{i+1}-3f_i+\frac{f_{i-1}^2+f_{i+1}^2}{f_{i-1}+f_{i+1}}+\frac{f_{i-1}^2+f_{i+1}^2-f_{i-2}f_{i-1}-f_{i+1}f_{i+2}}{f_{i-1}+f_{i+1}})\cdot \frac{f_{i-1}+f_{i+1}}{f_i^2}\\
&\geq&2(\frac{3}{2}(f_{i-1}+f_{i+1})-3f_i+\frac{f_{i-1}^2+f_{i+1}^2-f_{i-2}f_{i-1}-f_{i+1}f_{i+2}}{f_{i-1}+f_{i+1}})\cdot \frac{f_{i-1}+f_{i+1}}{f_i^2}\\
&=&\frac{3}{4}K_{g(h)_i}(K_{g(h)_i}+4)+2\frac{f_{i-1}^2+f_{i+1}^2-f_{i-2}f_{i-1}-f_{i+1}f_{i+2}}{f_i^2}
\end{eqnarray*}
For $i=n$, suppose the solution is complete then from Theorem \ref{LM1} $f_{n-1}>f_{n-2}$ and $f_{n+1}>f_{n+2}$, then
\begin{eqnarray*}
\triangle_{g(h)_{n}}K_{g(h)_{n}}&\geq&\frac{3}{4}K_{g(h)_{n}}(K_{g(h)_{n}}+4)
+2\frac{f_{n-1}^2+f_{n+1}^2-f_{n-2}f_{n-1}-f_{n+1}f_{n+2}}{f_{n}^2}\\
&\geq&\frac{3}{4}K_{g(h)_{n}}(K_{g(h)_{n}}+4).
\end{eqnarray*}
For $i=n-1$, $r\geq 6$ in (\ref{todathmPreliminary}), $r\geq 10$ in (\ref{vtoda}), suppose the solution is complete then from Theorem \ref{LM1} $f_{n-2}>f_{n-3}$ and $f_{n}\geq f_{n+1}$, then
\begin{eqnarray*}
\triangle_{g(h)_{n-1}}K_{g(h)_{n-1}}&\geq&\frac{3}{4}K_{g(h)_{n-1}}(K_{g(h)_{n-1}}+4)
+2\frac{f_{n-2}^2+f_{n}^2-f_{n-3}f_{n-2}-f_nf_{n+1}}{f_{n-1}^2}\\
&\geq&\frac{3}{4}K_{g(h)_{n-1}}(K_{g(h)_{n-1}}+4).
\end{eqnarray*}

Lastly, the lower bound of $K_{g(h)_i}$ are obvious from their formulas.
\end{proof}
Recall from Section \ref{pre} a solution to the Toda system (\ref{todathmPreliminary}) or the variant Toda system (\ref{vtoda}) gives a harmonic map $f:\mathbb{D}\rightarrow N=SL(r,\mathbb C)/SU(r)$. (In fact, the image of $f$ lies $SL(r,\mathbb R)/SO(r)$, which is a totally geodesic submanifold in $SL(r,\mathbb C)/SU(r)$.) If we suppose $r\geq 3$ in (\ref{todathmPreliminary}) and $r\geq 4$ in (\ref{vtoda}), then $f$ is immersed and conformal. We denote $e_f$ as the energy density, $g_f$ as the pullback metric and $K_{\sigma}^N$ as the sectional curvature of a tangent plane $\sigma$ in $TN$, where $\sigma$ is the image of the tangent map at a point on $\mathbb{D}$. By direct calculation (or see \cite{DaiLi16}), we have the formula:
\begin{lem}\label{gK}
\begin{eqnarray*}
&&e_f=2r\sum\limits_{i=0}^{r-1}g(h)_i,\quad g^{1,1}_f=e_f, \\
&&K_{\sigma}^N=-\frac{1}{2r}\frac{\sum\limits_{i=1}^{r}(f_
{i-1}-f_i)^2}{(\sum\limits_{i=1}^{r}f_i)^2} \text{ for (\ref{todathmPreliminary})}, \\ &&K_\sigma^N=-\frac{1}{2r}\frac{2(f_0-f_1)^2+2(f_0+f_1-f_2)^2+\sum\limits_{i=3}^{r-2}(f_{i-1}-f_i)^2}{(\sum\limits_{i=1}^{r}f_{i-1})^2}\text{ for (\ref{vtoda})}.
\end{eqnarray*}
\end{lem}
If $r\geq 3$ in (\ref{todathmPreliminary}) or $r\geq 4$ in (\ref{vtoda}), then $f$ is conformal, the pullback metric $g_f$ coincides with the energy density $e_f$.

From Theorem \ref{LM1} and Theorem \ref{varianttoda1}, for a complete solution in (\ref{todathmPreliminary}) or a complete real solution in (\ref{vtoda}), the induced curvature of the associated harmonic map is strictly negative or constantly zero.
The latter case cannot happen since we consider Higgs bundles on the unit disk $\mathbb D.$

For the variant Toda system, we obtain the following estimate of $f_0/f_1$ under a weaker completeness condition.
\begin{lem}\label{easycase} Let $r\geq 4$ in (\ref{vtoda}). If $g(h)_1$ is complete, then $f_0/f_1=\frac{e^{w_1+w_2}|q|_g^2}{e^{-w_1+w_2}}<1$ or constantly $1$.
\end{lem}
\begin{proof}
Away from zeros of $q$,  $\triangle_g\log (f_0/f_1)\geq 2(f_0-f_1),$
$$\triangle_g (f_0/f_1)\geq (f_0/f_1)\cdot\triangle_g \log(f_0/f_1) =2(f_0/f_1)(f_0-f_1).$$
Since both sides are smooth, it extends to the whole surface.
So we have $$\triangle_{g(h)_1}\log (f_0/f_1)=\triangle_{g\cdot f_1}\log (f_0/f_1)\geq 2(f_0/f_1-1),$$
Then by the Cheng-Yau maximum principle and the assumption that $g(h)_1$ is complete, we obtain $f_0/f_1\leq 1.$
By the strong maximum principle, either $f_0/f_1<1$ or $f_0/f_1\equiv 1.$
\end{proof}

Now we study the boundedness of the geometric objects about the harmonic map. For simplicity, we first assume $(\Sigma,g)=(\mb{D},g_{\mb{D}})$. For the upper bound, in \cite{LiMochizuki}, Li-Mochizuki showed that for any Higgs bundle with a harmonic metric, the boundedness of the spectrum of the Higgs field is equivalent to the upper boundedness of the energy density of the corresponding harmonic map. In particular, applying this theorem to the equation system (\ref{todathmPreliminary}) and (\ref{vtoda}), we obtain
\begin{thm}(\cite[Proposition 3.12]{LiMochizuki})\label{LM2}
Consider the Toda system (\ref{todathmPreliminary}) $r\geq 2$ and the variant Toda system (\ref{vtoda}) $r\geq 3$ on $(\mb{D},g_{\mb{D}})$. Then $q$ is bounded with respect to $g_{\mb{D}}$ if and only if $e_f=\sum\limits_{i=0}^{r-1}g(h)_i=(\sum\limits_{i=0}^{r-1}f_i)g_{\mb{D}}\leq Cg_{\mb{D}}$ for some constant $C$. Clearly they are also equivalent to $g(h)_i=f_ig_{\mb{D}}\leq Cg_{\mb{D}}$, $0\leq i\leq n$ for some constant $C$.
\end{thm}
If $g(h)_i$ is complete, we have the following result.
\begin{prop}(\cite[Proposition 3.20]{LiMochizuki} for (\ref{todathmPreliminary}) and \cite{Sagman} for (\ref{vtoda}))\label{onecom}
Consider the Toda system (\ref{todathmPreliminary}) $r\geq 2$ and the variant Toda system (\ref{vtoda}) $r\geq 3$. Suppose $g(h)_{i}$ is complete for some $i\in\{1,\cdots,n\}$. Then there exists a constant $C>0$ such that $g(h)_j\leq Cg(h)_{i}$ for every $j=0,\cdots,n$ and $|q|_{g(h)_i}\leq C$.
\end{prop}
For the lower bound of the energy density, we have the following estimate.
\begin{lem}\label{LowerBound}
Consider the Toda system (\ref{todathmPreliminary}) $r\geq 2$ and the variant Toda system (\ref{vtoda}) $r\geq 3$ on $(\mb{D},g_{\mb{D}})$. Suppose $g(h)_i$ is complete for some $i\in\{1,\cdots,n\}$, then there exists a constant $C>0$ such that
$g(h)_i\geq C^{-1}g_{\mathbb D}$. In particular, there exists a constant $C$ such that $e_f\geq C^{-1}g_{\mathbb D}.$
\end{lem}
\begin{proof}
We have shown that in each case, for $1\leq i\leq n$, $f_i$ satisfies the estimate $\triangle_{g_{\mb{D}}}\log f_i=af_i-bf_{i-1}-cf_{i+1}-\frac{1}{2}$ for some constant $a,b,c\geq 0$. So
$$\triangle_{g(h)_i}\log f^{-1}_i=-\frac{1}{f_i}(af_i-bf_{i-1}-cf_{i+1}-\frac{1}{2})\geq \frac{1}{2}f^{-1}_{i}-a.$$
From the expression of $K_{g(h)_i}$, it is automatically bounded below. Then from the Cheng-Yau maximum principle, we obtain $f_i\geq \frac{1}{2a}$.
\end{proof}

Combining Theorem \ref{LM2}, Proposition \ref{onecom} and Lemma \ref{LowerBound} together, we obtain
\begin{prop}\label{UpperBound}
Consider the Toda system (\ref{todathmPreliminary}) $r\geq 2$ or the variant Toda system (\ref{vtoda}) $r\geq 3$ on $(\mb{D},g_{\mb{D}})$.  Let $(w_1,\cdots, w_n)$ be a solution.
Then the followings are equivalent:

(1) $q$ is bounded with respect to $g_{\mathbb D}$;

(2) There is a constant $C>0$ such that $e_f\leq C g_{\mb{D}}$;

(3) There is a constant $C>0$ such that $g(h)_i\leq C g_{\mb{D}}$ for all $0\leq i\leq n$.

If in addition $g(h)_{i_0}$ is complete for some $i_0\in\{1,\cdots,n\}$, then the conditions above are also equivalent to

(4) There is a constant $C>0$ such that $g(h)_{i_0}\leq C g_{\mb{D}}$.
\end{prop}
\begin{proof}
We only need to show (3)$\Rightarrow$(1): Since $g(h)_i\leq Cg_{\mb{D}}$, we have $e^{-w_{i}+w_{i+1}}\leq C$ for $i=0,\cdots,n$. Hence
$$|q|^2_{g_{\mb{D}}}=\big(|q|^2_{g_{\mb{D}}}e^{2w_1}\big)\big(e^{-w_1+w_2}\big)^2\cdots\big(e^{-w_{n-1}+w_n}\big)^2
\big(e^{-(2n+2-r)w_n}\big)^{\frac{2}{2n+2-r}}\leq C.$$
\end{proof}

\begin{lem}\label{ModelEquationEstimate}
Consider the Toda system (\ref{todathmPreliminary}) for $r\geq 2$ and the variant Toda system (\ref{vtoda}) for $r\geq 3$ on $(\mb{D},g_{\mb{D}})$. Suppose $g(h)_i$ is complete for some $i\in\{1,\cdots,n\}$. Suppose there exist constants $c,d>0$ such that
\begin{equation}
    \triangle_{g(h)_i}K_{g(h)_i}\geq cK_{g(h)_i}(K_{g(h)_i}+d).
\end{equation}
Then $K_{g(h)_i}<0$. Moreover if $g(h)_i\leq Cg_{\mathbb D}$ for some constant $C>0$, then there is a constant $\delta>0$ such that
$K_{g(h)_i}\leq -\delta.$
\end{lem}
\begin{proof}
Since $g(h)_i$ is complete and $K_{g(h)_i}$ has a lower bound automatically, it follows from Lemma \ref{KEYLEMMA}.
\end{proof}

Now we prove our main theorem, which generalizes Wan's result \cite{WAN} to higher rank case and extends Theorem \ref{LM2} to relate the boundedness of the solution to more geometric objects.

\begin{thm}\label{cyclicbound}
Consider the Toda system (\ref{todathmPreliminary}) for $r\geq 3$ and the variant Toda system (\ref{vtoda}) for $r\geq 4$. Let $(w_1,\cdots, w_n)$ be the complete solution. Then the followings are equivalent:

(1) $q$ is bounded with respect to $g_{\mathbb D}$;

(2) $|w_i|,~i=1,\cdots, n$ are bounded;

(3) There is a constant $C>0$ such that $C^{-1}g_{\mb{D}}\leq g(h)_i\leq C g_{\mb{D}}$ for all $i\in\{1,\cdots,n\}$;

(3') There is a constant $C>0$ such that $C^{-1}g_{\mb{D}}\leq g(h)_{i_0}\leq C g_{\mb{D}}$ for some $i_0\in\{1,\cdots,n\}$;

(4) There is a constant $C>0$ such that $C^{-1}g_{\mb{D}}\leq g_f\leq C g_{\mb{D}}$;

(5) The curvature $K_{g_f}$ of the pullback metric $g_f$ is bounded above by a negative constant.

(6) The curvature $K_{\sigma}^N$ is bounded above by a negative constant;

(7) There is a constant $\delta>0$ such that $g(h)_{i-1}\leq (1-\delta) g(h)_{i}$ for every $i=1,\cdots,n$, in the variant Toda system $(\ref{vtoda})$ $g(h)_1\leq (1-\delta) g(h)_{2}$ is replaced by $g(h)_0+g(h)_1\leq (1-\delta) g(h)_{2}$;

(7') There is a constant $\delta>0$ such that $g(h)_{i_0-1}\leq (1-\delta) g(h)_{i_0}$ for some $i_0\in\{2,\cdots,n\}$, in the variant Toda system $(\ref{vtoda})$ $g(h)_1\leq (1-\delta) g(h)_{2}$ is replaced by $g(h)_0+g(h)_1\leq (1-\delta) g(h)_{2}$;

(8) For $r\geq4$, $r\neq 5$ in (\ref{todathmPreliminary}), $r\geq 6$ in (\ref{vtoda}), the curvature $K_{g(h)_{n-1}}$ is bounded above by a negative constant.

(9) For $r\geq3$ in (\ref{todathmPreliminary}), $r\geq 4$ in (\ref{vtoda}), the curvature $K_{g(h)_{n}}$ is bounded above by a negative constant.
\end{thm}
\begin{proof}
Notice that in our case $f$ is conformal, so $g_f=e_f$. From Lemma \ref{LowerBound}, we see $f_i=e^{-w_i+w_{i+1}}$, $i=1,\cdots,n$ have lower bound away from zero. So we automatically have the lower bound of $g(h)_i$ and $g_f$. From the lower bound of $f_n=e^{-(2n+2-r)w_n}$, we have the upper bound of $w_n$. Then from the induction we obtain the upper bound of all $w_i$, $i=1,\cdots,n$.

By the same argument the upper boundedness of all $f_i$, $i=1,\cdots,n$ implies the lower boundedness of all $w_i$, $i=1,\cdots,n$. Then by Proposition \ref{UpperBound}, (1)(2)(3)(3')(4) are equivalent.

Next we show (6)$\Rightarrow$(5)$\Rightarrow$(4).

(6)$\Rightarrow$(5): From the Gauss equation, $K_{g_f}=K_{\sigma}^N+\det(II)$, where $II$ is the second fundamental form under an orthonormal basis. Since $f$ is harmonic and conformal, $f$ is minimal. So $\det(II)\leq 0$. So $K_{\sigma}^N\leq -\delta$ implies $K_{g_f}\leq -\delta.$

(5)$\Rightarrow$(4): Since $f$ is conformal, it follows from Lemma \ref{KC}.

Next, we show (3)$\Rightarrow$(8)$\Rightarrow$(6), (3)$\Rightarrow$(9)$\Rightarrow$(6).

(3)$\Rightarrow$(8), (3)$\Rightarrow$(9): 
From the conditions in (3), $g(h)_{n-1}$ or $g(h)_n$ is upper bounded by $g_{\mb D}$.  Applying Lemma \ref{ModelEquationEstimate} and Lemma \ref{Kn-1}, we obtain the estimate.

(8)$\Rightarrow$(6), (9)$\Rightarrow$(6): We only prove (8)$\Rightarrow$(6). The proof of (9)$\Rightarrow$(6) is similar. Note that by Proposition \ref{onecom}, $f_j$, $j=0,\cdots,n$ are bounded by $Cf_{n-1}$ for some positive constant $C$. Then from Lemma \ref{gK},\\
(i) for (\ref{todathmPreliminary}):
$$K_\sigma^N=-\frac{1}{2r}\frac{\sum_{i=1}^r(f_{i-1}-f_i)^2}{(\sum_{i=1}^{r}f_{i-1})^2}\leq-\frac{1}{2rC}\frac{(f_{n-2}-f_{n-1})^2+(f_{n}-f_{n-1})^2}{f_{n-1}^2}\leq -\frac{1}{4rC}\frac{(f_{n-2}+f_n-2f_{n-1})^2}{f_{n-1}^2}\leq -\delta.$$\\
(ii) for (\ref{vtoda}):
$$K_\sigma^N=-\frac{1}{2r}\frac{2(f_0-f_1)^2+2(f_0+f_1-f_2)^2+\sum\limits_{i=3}^{r-2}(f_{i-1}-f_i)^2}{(\sum\limits_{i=1}^{r}f_{i-1})^2}$$
So if $n\geq 4$, $$K_\sigma^N\leq-\frac{1}{2rC}\frac{(f_{n-2}-f_{n-1})^2+(f_{n}-f_{n-1})^2}{f_{n-1}^2}\leq -\frac{1}{4rC}\frac{(f_{n-2}+f_n-2f_{n-1})^2}{f_{n-1}^2}\leq -\delta.$$
If $n=3$, $$K_\sigma^N\leq-\frac{1}{2rC}\frac{(f_0+f_1-f_2)^2+(f_2-f_3)^2}{f_2^2}\leq -\frac{1}{4rC}\frac{(f_0+f_1+f_3-2f_2)^2}{f_2^2}\leq -\delta.$$

Finally, we show (2)$\Rightarrow$(7)$\Rightarrow$(7')$\Rightarrow$(6).

(2)$\Rightarrow$(7): For each $i\in\{1,\cdots,n\}$, define the metric $g_i=e^{-\frac{2}{r+1-2i}w_i}g_{\mb{D}}$. Since $|w_i|$ is bounded, $g_i$ is complete. We calculate the curvature of $g_i$, denote $g_{\mb{D}}=\ti{g}_{\mb{D}}dz\otimes d\bar{z}$,
\begin{\eq}
K_{g_i}&=&-\frac{2}{e^{-\frac{2}{r+1-2i}w_i}\ti{g}}\partial_z\partial_{\bar{z}}\log(e^{-\frac{2}{r+1-2i}w_i}\ti{g}_{\mb{D}})\\
&=&-2e^{\frac{2}{r+1-2i}w_i}(-\frac{2}{r+1-2i}\triangle_{g_{\mb{D}}}w_i+\triangle_{g_{\mb{D}}}\log \ti{g}_{\mb{D}})\\
&=&-2e^{\frac{2}{r+1-2i}w_i}(-\frac{2}{r+1-2i}\triangle_{g_{\mb{D}}}w_i+\frac{1}{2}).
\end{\eq}
Then from (\ref{todathmPreliminary}) and (\ref{vtoda}), set $w_0=-w_1-\log|q|_{g}^2$, $w_{n+1}=-(2n+1-r)w_n$ in (\ref{todathmPreliminary}) and $w_0=-w_2-\log|q|_{g}^2$, $w_{n+1}=-(2n+1-r)w_n$ in (\ref{vtoda}), except $i=2$ in (\ref{vtoda}), we have
\begin{\eq}
K_{g_i}=\frac{4e^{\frac{2}{r+1-2i}w_i}}{r+1-2i}(e^{-w_{i-1}+w_i}-e^{-w_i+w_{i+1}}),\quad 1\leq i\leq n.
\end{\eq}
For $i=2$ in (\ref{vtoda}), $K_{g_2}=\frac{4e^{\frac{2}{r-3}w_2}}{r-3}(e^{-w_{0}+w_1}+e^{-w_{1}+w_2}-e^{-w_2+w_{3}}).$
From Theorem \ref{LM1} and Theorem \ref{varianttoda1}, $K_{g_i}<0$.
From the assumption of (2), $g_i$ is equivalent to $g_{\mb{D}}$.
Now set $$u_i=1-e^{-w_{i-1}+w_i}/e^{-w_i+w_{i+1}},~1\leq i\leq n \text{ except }u_2=1-(e^{-w_{0}+w_1}+e^{-w_1+w_2})/e^{-w_2+w_{3}}\text{ for $(\ref{vtoda})$}.$$ Then $u_i>0$. From the assumption (2), $|w_i|$ are bounded and thus $u_i$ is mutually bounded by $-K_{g_i}$.

To apply Lemma \ref{Wan}, we calculate the Bochner formula for $u_i$, $i=1,\cdots,n$.
\\
(i) For (\ref{todathmPreliminary}), from the formula (\ref{1})(\ref{2}) in Lemma \ref{n=1}, we obtain
\begin{\eq}
\triangle_{g_\mb{D}} u_1&\leq& -e^{3w_1-w_2}|q|^2_{g_\mb{D}}\triangle_{g_\mb{D}} (3w_1-w_{2}-\frac{r}{2}),\\
\triangle_{g_\mb{D}} u_i&\leq& -e^{-w_{i-1}+2w_i-w_{i+1}}\triangle_{g_\mb{D}} (-w_{i-1}+2w_i-w_{i+1}),\quad 2\leq i\leq n.
\end{\eq}
Then
\begin{\eq}
\triangle_{g_\mb{D}} u_1&\leq& e^{3w_1-w_2}|q|^2_{g_\mb{D}} (3e^{-w_1+w_2}u_1-e^{-w_2+w_{3}}u_2),\\
\triangle_{g_\mb{D}} u_i&\leq& e^{-w_{i-1}+2w_i-w_{i+1}} (-e^{-w_{i-1}+w_{i}}u_{i-1}+2e^{-w_{i}+w_{i+1}}u_i-e^{-w_{i+1}+w_{i+2}}u_{i+1}),\quad 2\leq i\leq n.
\end{\eq}
From the assumption, $|w_i|$'s and then $|q|^2_{g_{\mb{D}}}$ are bounded. So we obtain that there is a constant $c>0$ depending on $C$ such that
$$\triangle_{g_i} u_i=e^{\frac{2}{r+1-2i}w_i}\triangle_{g_{\mb{D}}} u_i\leq c u_i,\quad i=1,\cdots,n.$$
Then Lemma \ref{Wan} implies $u_i\geq \delta$ for some constant $\delta>0$ depending on $C$, which means $g(h)_{i-1}\leq (1-\delta) g(h)_{i}$.
\\
(ii) For (\ref{vtoda}), similarly,
\begin{\eq}
\triangle_{g_\mb{D}} u_1&\leq& -e^{2w_1}|q|^2_{g_\mb{D}}\triangle_{g_\mb{D}} (2w_1-\frac{r}{2}),\\
\triangle_{g_\mb{D}} u_2&\leq& -e^{w_1+2w_2-w_3}|q|^2_{g_\mb{D}}\triangle_{g_\mb{D}} (w_1+2w_2-w_3-\frac{r}{2})\\
&&-e^{-w_{1}+2w_2-w_{3}}\triangle_{g_\mb{D}} (-w_{1}+2w_2-w_{3}),\\
\triangle_{g_\mb{D}} u_i&\leq& -e^{-w_{i-1}+2w_i-w_{i+1}}\triangle_{g_\mb{D}} (-w_{i-1}+2w_i-w_{i+1}),\quad 2\leq i\leq n.
\end{\eq}
Then
\begin{\eq}
\triangle_{g_\mb{D}} u_1&\leq& e^{2w_1}|q|^2_{g_\mb{D}}(2e^{-w_{1}+w_{2}}u_1-\frac{r-1}{2}),\\
\triangle_{g_\mb{D}} u_2&\leq& e^{w_1+2w_2-w_3}|q|^2_{g_\mb{D}}(e^{-w_{1}+w_{2}}u_1-\frac{r-1}{4}+2e^{-w_{2}+w_{3}}u_2-2e^{w_1+w_2}|q|^2_{g_{\mb{D}}}-e^{-w_{3}+w_{4}}u_{3})\\
&&+e^{-w_{1}+2w_2-w_{3}} (-e^{-w_{1}+w_{2}}u_{1}+2e^{-w_{2}+w_{3}}u_2-2e^{w_1+w_2}|q|^2_{g_{\mb{D}}}-e^{-w_{3}+w_{4}}u_{3})\\
&\leq&(e^{w_1}|q|^2_{g_\mb{D}}-e^{-w_1})e^{2w_2-w_{3}}e^{-w_{1}+w_{2}}u_1+cu_2,\\
\triangle_{g_\mb{D}} u_3&\leq& e^{-w_{2}+2w_3-w_{4}} (e^{w_1+w_2}|q|^2_{g_{\mb{D}}}+e^{-w_1+w_2}-e^{-w_2+w_3}+2e^{-w_{3}+w_{4}}u_3-e^{-w_{4}+w_{5}}u_{4}),\\
\triangle_{g_\mb{D}} u_i&\leq& e^{-w_{i-1}+2w_i-w_{i+1}} (-e^{-w_{i-1}+w_{i}}u_{i-1}+2e^{-w_{i}+w_{i+1}}u_i-e^{-w_{i+1}+w_{i+2}}u_{i+1}),\quad 4\leq i\leq n.
\end{\eq}
From Theorem \ref{varianttoda1}, $e^{w_1+w_2}|q|^2_{g_{\mb{D}}}<e^{-w_1+w_2}$, $e^{w_1+w_2}|q|^2_{g_{\mb{D}}}+e^{-w_1+w_2}<e^{-w_2+w_3}$, so we obtain $\triangle_{g_{\mb{D}}} u_i\leq c u_i$, $i=1,\cdots,n.$ Then Lemma \ref{Wan} implies the desired results.

(7)$\Rightarrow$(7'): It is obvious.

(7')$\Rightarrow$(6): Note that by Proposition \ref{onecom}, $\sum_{i=1}^rf_{i-1}$ is bounded by $Cf_{i_0}$ for some positive constant $C$. Then from the formula in Lemma \ref{gK} and the assumption in (7'),
\\
(i) For (\ref{todathmPreliminary}),
 $$K_\sigma^N=-\frac{1}{2r}\frac{\sum_{i=1}^r(f_{i-1}-f_i)^2}{(\sum_{i=1}^{r}f_{i})^2}\leq-\frac{1}{2rC}\frac{(f_{i_0-1}-f_{i_0})^2}{f_{i_0}^2}\leq -\delta.$$
(ii) For (\ref{vtoda}), the proof is similar.
\end{proof}




Notice that in Lemma \ref{Kn-1}, in some lower rank cases, to obtain the Bochner formula for curvatures, we only need the completeness of $g(h)_n$ or $g(h)_{n-1}$, but not the completeness of the whole solution. By using the similar method, we obtain the following results.
\begin{thm}\label{onecomn}
Consider the variant Toda system (\ref{vtoda}) for $r=4,5$. Let $(w_1, w_2)$ be a solution. Suppose $g(h)_2$ is complete. Then the followings are equivalent:

(1) $q$ is bounded with respect to $g_{\mathbb D}$;

(2) $w_i$, $i=1,2$ are lower bounded;

(3) There is a constant $C>0$ such that $g(h)_i\leq Cg_{\mb{D}}$ for all $i\in\{1,2\}$;

(3') There is a constant $C>0$ such that $g(h)_{2}\leq C g_{\mb{D}}$;

(4) There is a constant $C>0$ such that $g_f\leq C g_{\mb{D}}$;

(5) The curvature $K_{g_f}$ of the pullback metric $g_f$ is bounded above by a negative constant.

(6) The curvature $K_{\sigma}^N$ is bounded above by a negative constant;

(7) The curvature $K_{g(h)_{2}}$ is bounded above by a negative constant.
\end{thm}
\begin{thm}\label{onecomn-1}
Consider the Toda system (\ref{todathmPreliminary}) for $r=4$ and the variant Toda system (\ref{vtoda}) for $r=7$. Let $n=[\frac{r}{2}]$. Let $(w_1,\cdots, w_n)$ be a solution. Suppose $g(h)_{n-1}$ is complete. Then the followings are equivalent:

(1) $q$ is bounded with respect to $g_{\mathbb D}$;

(2) $w_i$, $i=1,\cdots, n$ are lower bounded;

(3) There is a constant $C>0$ such that $g(h)_i\leq Cg_{\mb{D}}$ for all $i\in\{1,\cdots, n\}$;

(3') There is a constant $C>0$ such that $g(h)_{n-1}\leq C g_{\mb{D}}$;

(4) There is a constant $C>0$ such that $g_f\leq C g_{\mb{D}}$;

(5) The curvature $K_{g_f}$ of the pullback metric $g_f$ is bounded above by a negative constant.

(6) The curvature $K_{\sigma}^N$ is bounded above by a negative constant;

(7) The curvature $K_{g(h)_{n-1}}$ is bounded above by a negative constant.
\end{thm}
So far we considered $(\Sigma,g)=(\mb{D},g_{\mb{D}})$. In fact this constraint is not essential since the geometric objects we consider are independent of the choice of the K\"ahler metric on $\Sigma$.
\begin{cor}
Let $\Sigma$ be a Riemann surface with a hyperbolic Hermitian metric $g_{-1}$. Let $g$ be a Riemannian metric conformal to $g_{-1}$ with bounded conformal factor.
Then Theorem \ref{cyclicbound}, Theorem \ref{onecomn} and Theorem \ref{onecomn-1} hold for replacing $(\mb{D},g_{\mb{D}})$ by $(\Sigma,g)$.
\end{cor}
\begin{proof}
From Remark \ref{equiv}, the solutions with respect the metric $g_{-1}$ and $g$ are bounded to each other. So the boundedness in the theorems above are consistent under $g_{-1}$ and $g$. Considering the universal cover also don't change the boundedness. For the geometric objects, they come from the harmonic map, which is only depend on the complex structure of $\Sigma$. So they are independent of the choice of $g$ or  $g_{-1}$.
\end{proof}
\begin{rem}
Suppose $g$ is complete and conformal to the hyperbolic metric. Then from the Cheng-Yau maximum principle, the pinched hyperbolic condition $-a\leq K_g\leq -b$ for constants $a,b>0$ implies the conformal factor is bounded.
\end{rem}

\section{Harmonic maps between surfaces and bounded quadratic differentials}\label{r=2}

In this section, we discuss the Toda system for $r=2$, which corresponds to the harmonic map equation between surfaces. First we recall some calculations in \cite{SchoenYau}. Let $(\Sigma,g=\sigma(z)|dz|^2)$, $(M,h=\mu(u)|du|^2)$ be two Riemann surfaces with K\"ahler metrics. Let $f$ be a harmonic map between $\Sigma$ and $M$.

Set $H=|\partial u|_{g,h}^2=|u_z|^2\frac{\mu}{\sigma}$, $L=|\bar{\partial} u|_{g,h}^2=|u_{\bar{z}}|^2\frac{\mu}{\sigma}$. Then the energy density $e(f):=\frac{1}{2}|df|_{g,h}^2=H+L$. The Jacobian $J(f)=H-L$. The Hopf differential is the $(2,0)$-part of the pullback metric $\text{Hopf}(f)=u_z\bar{u}_z\mu dz\otimes dz$, denoted as $q$, so $|q|_{g}^2=HL$. Denote by $K_{g},K_{h}$ the Gaussian curvature of $\Sigma, M$ respectively. Then at nonzero point of $H$, (our $\triangle_g$ differs from the notation in \cite{SchoenYau} by a factor $4$)
\begin{equation*}\label{H}
4\triangle_{g}\log H=-2K_{h} H+2K_{h} L+2K_{g}.
\end{equation*}

Let $w=-\frac{1}{2}\log H$, 
 the above equation becomes
\begin{equation}\label{harmoniceq}
\triangle_g w=-\frac{K_h}{4}(|q|_g^2e^{2w}-e^{-2w})-\frac{K_{g}}{4},
\end{equation}
which coincides with equation (\ref{harmonic}) for $\Sigma=M=\mb{D}$, $g=4h=g_{\mb{D}}$.
Recall in \cite{WAN} the map $f$ is called orientation-preserving if the Jacobian $J(f)=H-L=e^{-2w}-|q|_g^2e^{2w}\geq 0$.
The map $f$ is called quasi-conformal if $\frac{L}{H}=|q|_g^2e^{4w}\leq k$ for some constant $k<1$. The energy density of $f$ is given by $e^{-2w}+|q|_g^2e^{2w}$. Let $|\partial f|^2=H\cdot g.$

Wan in \cite{WAN} showed the following result.
\begin{thm}(\cite{WAN}) \label{WAN}
Let $f$ be an orientation-preserving harmonic map from $\mathbb D$ to itself. Suppose the metric $|\partial f|^2$ is complete.
Then the followings are equivalent:\\
(a) the Hopf differential is bounded with respect to $g_{\mathbb D}$;\\
(b) $f$ is quasi-conformal;\\
(c) the energy density of $f$ is bounded.

In such cases $f$ is a diffeomorphism.
\end{thm}

The universal Teichm\"uller space is the space of quasi-symmetric homeomorphism equipped with $C^0$ topology between $S^1$ fixing three points. By the work of \cite{LiTam} and \cite{Markovic}, equivalently, the universal Teichm\"uller space $\mathcal T(\mathbb D)$ is in bijection with the space of harmonic quasi-conformal homeomorphisms between $\mathbb D$ up to $PSL(2,\mathbb R)$-action. Therefore, by Theorem \ref{WAN}, there is a bijection between the space of bounded quadratic differentials with the universal Teichm\"uller space $\mathcal T(\mathbb D)$.

Li, Tam and Wang in \cite{LiTW} generalized Wan's result to more general surfaces instead of $\mb{D}$ using a similar technique. The hyperbolic Hadamard surfaces are complete, simply connected, Riemannian surfaces with Gaussian curvature $K$ satisfying $-\kappa\leq K\leq 0$ for some constant $\kappa>0$ and have positive lower bounds for their spectra. It is shown in  \cite{LiTW} that the metric of a hyperbolic Hadamard surface is equivalent to the uniformization hyperbolic metric, i.e. the conformal factor is bounded away from $0$ and $\infty$.
\begin{thm}(\cite{LiTW})\label{LiTW}
Let $S_1$, $S_2$ be two hyperbolic Hadamard surfaces. Let $f:S_1\rightarrow S_2$ be a harmonic diffeomorphism. Then the followings are equivalent:\\
(a) the Hopf differential is bounded with respect to the uniformization hyperbolic metric;\\
(b) $f$ is quasi-conformal;\\
(c) the energy density of $f$ is bounded.
\end{thm}
In Theorem \ref{LiTW}, the assumption of the domain is not essential since the harmonicity only depends on the conformal structure. In fact, we only need to consider $\mb{D}$. From \cite{SchoenRole}, $f$ being a diffeomorphism implies the metric $|\partial f|^2$ being complete.

From Proposition \ref{n=1 bd}, under certain condition, the assumption on the diffeomorphism in Theorem \ref{LiTW}  or orientation-preserving in Theorem \ref{WAN} can be replaced by the completeness of the metric $|\partial f|^2$.
\begin{thm}\label{harm}
Let $(\Sigma,g)$ and $(M,h)$ be two surfaces with Riemannian metrics.  Suppose $g$ is conformal and mutually bounded to a hyperbolic metric. Suppose the curvature of $h$ satisfies $-C_1\leq K_h\leq -C_2$ for some constants $C_1\geq C_2>0$. Let $f$ be a harmonic map from $(\Sigma,g)$ to $(M,h)$. Suppose $|\partial f|^2$ is complete. Then the followings are equivalent:\\
(a) the Hopf differential is bounded with respect to the uniformization hyperbolic metric;\\
(b) $f$ is quasi-conformal;\\
(c) the energy density of $f$ is bounded.

In such cases $f$ is a diffeomorphism.
\end{thm}
\begin{proof}
The harmonic equation (\ref{harmoniceq}) is in the form of equation (\ref{ab}) and satisfies the assumptions in Proposition \ref{n=1 bd}.  Notice that the quasi-conformality and the energy density only depends on the conformal structure of $g$. Since $g$ is conformal and mutually bounded to the hyperbolic metric, we only need to consider the boundedness under the hyperbolic metric. Then by considering the universal cover, Theorem \ref{harm} follows from Proposition \ref{n=1 bd}.
\end{proof}

\section{Hyperbolic affine spheres and bounded cubic differentials}\label{r=3}



In this section, we discuss the relation between bounded cubic differentials and hyperbolic affine spheres in $\mathbb R^3$. First, we review some backgrounds on hyperbolic affine spheres. For more details, one can refer \cite{Loftin0, BenoistHulin, DumasWolf}.

For a non-compact simply connected $2$-manifold $M$, consider a locally strictly convex immersed hypersurface $f:M\rightarrow \mathbb{R}^3$. Affine differential geometry associates to such a locally convex hypersurface a transversal vector field $\xi$, called the affine normal. Being an affine spherical immersion means the affine normal meets at a point (the center). By applying a translation, we can move the center of the affine sphere to the origin and write $\xi(p)=-Hp$, for all $p\in f(M)\subset \mathbb{R}^3$ for some constant $H\in \mathbb{R}$, the affine curvature. In the case when $H$ is negative, call the affine spherical immersion hyperbolic. After renormalization, we obtain a hyperbolic affine spherical immersion with center $0$ and of affine curvature $-1$, called normalized hyperbolic affine sphere.

Decomposing the standard connection $D$ of $\mathbb{R}^3$ into tangent direction of $f(M)$ and affine normal components:
\[D_XY=\nabla_XY+h(X,Y)\xi, \quad \forall X,Y\in T_{f(p)}f(M).\]
The second fundamental form $h$ of the image $f(M)$ relative to the affine normal $\xi$ can define a Riemannian metric $h$ on $M$, the \textit{Blaschke metric}. This induces a complex structure on $M$. Also, the decomposition defines an induced connection $\nabla$ on $TM$. Let $\nabla^h$ be the Levi-Civita connection of the Blaschke metric $h$ and the Pick form $A(X,Y,Z)=h((\nabla-\nabla^h)_XY,Z)$ is a $3$-tensor, which uniquely determines a cubic differential $q=q(z)dz^3$ such that $\Re q=A$, the \textit{Pick differential}.

We call an affine sphere complete if its Blaschke metric is complete.
By the work of Cheng-Yau \cite{ChengYau77, ChengYau86} and An-Min Li \cite{Li92}, there is a correspondence between properly convex domains in $\mathbb RP^2$ and complete hyperbolic affine spheres. An open subset $\Omega\subset \mathbb RP^2$ is properly convex if it corresponds to a bounded convex domain in $\mathbb R^2$ when restricted to an affine chart. For a properly convex domain $\Omega$ in $\mathbb RP^2$, denote by $C(\Omega)$ one of the two open convex cones above $\Omega$. 
Given a properly convex domain $\Omega\subset\mathbb RP^2$, there exists a unique normalized hyperbolic affine sphere, which is asymptotic to the boundary of the cone $C(\Omega)$. Moreover, the Blaschke metric on the affine sphere is complete. Conversely, every normalized complete hyperbolic affine sphere is asymptotic to the boundary of a cone above a properly convex subset of $\mathbb RP^2$.

Here we focus on the affine spherical immersions where $(M,h)$ is conformal to the hyperbolic disk $(\mathbb{D}, g_{\mathbb D}=\sigma(z)|dz|^2)$. In this case, we can reparametrize the hyperbolic affine spherical immersion by $f:\mathbb{D}\rightarrow \mathbb{R}^3$. Write the Blaschke metric $h=e^w\cdot g_{\mathbb D}$ and $q=q(z)dz^3$. Following Wang \cite{Wang} and Simon-Wang \cite{SimonWang}, $f$ being an affine spherical immersion is equivalent to $q$ being holomorphic and $(w,q)$ satisfying
\begin{equation}\label{WangEquation}
\triangle_{g_{\mathbb D}}w=2e^w-4|q|_{g_{\mathbb  D}}^2e^{-2w}-2.\end{equation}
Up to a change of constants, this coincides with the Toda equation in $r=3$ case for a real solution.
The curvature of the Blaschke metric $h=e^w\cdot g_{\mathbb D}$ is \[k_h=-\frac{1}{2}\triangle_h w=-1+2|q|_{g_{\mathbb D}}^2e^{-3w}.\]



Then Lemma \ref{n=1} implies that the curvature of the Blaschke metric of a complete affine sphere is non-positive, which was proved by Calabi \cite{Calabi}.

As a direct corollary of Proposition \ref{n=1 bd} to $r=3$, we can recover the following theorem shown by Benoist and Hulin.
\begin{thm}(Benoist-Hulin \cite{BenoistHulin})\label{Rank3}
For a complete hyperbolic affine spherical immersion $f: \mathbb D\rightarrow \mathbb R^3$, the followings are equivalent:\\
(a) the Pick differential is bounded with respect to the hyperbolic metric;\\
(b) the affine metric has curvature bounded above by a negative constant;\\
(c) the affine metric is conformally bounded with respect to the hyperbolic metric.
\end{thm}



Let's briefly explain the original proof from (c) to (b) in Benoist-Hulin \cite{BenoistHulin}. On a properly convex domain, one can also define the Hilbert metric.  Their proof is by adding the following three equivalent conditions:\\
(d) $\Omega$ equipped with the affine metric is Gromov hyperbolic; \\
(e) $\Omega$ equipped with the Hilbert metric is Gromov hyperbolic;\\
(f) the orbit closure $\overline{SL(3,\mathbb R)\cdot \Omega}$ in $X$ does not contain the projective triangle.\\
In this way, the authors in \cite{BenoistHulin} show that for a properly convex domain $\Omega$, the Pick differential is bounded with respect to the hyperbolic metric if and only if $\Omega$ is Gromov hyperbolic with respect to the Hilbert metric. One may view the moduli space of Gromov hyperbolic convex sets (in Hilbert metric or affine metric) as a generalization of the universal Teichm\"uller space to rank $3$.

The proof of (c)$\Rightarrow$(b) in \cite{BenoistHulin} is by showing (c)$\Rightarrow$(d)$\Rightarrow$(e)$\Rightarrow$(f)$\Rightarrow$(b), which relies on
the following two facts:\\
(i) Continuity dependence of the curvature function: the curvature of the affine metric depends continuously on the pair $(x,\Omega)\in \mathcal E$, due to Benoist-Hulin \cite{BenoistHulinFiniteVolume}.\\
(ii)
Benz\'ecri compactness \cite{Benzecri}: Consider $\mathcal E$ the set of pairs $(x,\Omega)$ where $\Omega\subset \mathbb RP^2$ is a properly convex domain and $x$ is a point in $\Omega$. The natural action of $SL(3,\mathbb R)$ on the space of $\mathcal E$, equipped with the Hausdorff topology is compact. \\

(c) $\Rightarrow$ (d): it follows from that quasi-isometry preserves Gromov hyperbolicity.

(d) $\Rightarrow$ (e): it uses the fact the densities $\mu_{\text{Hilbert}}$ and $\mu_{\text{affine}}$ are uniformly bounded with respect to each other which follows from Benz\'ecri's compactness.

(e) $\Rightarrow$ (f): it follows from the fact that the limit of a sequence of Gromov $\delta$-hyperbolic spaces is still Gromov $\delta$-hyperbolic.

(f) $\Rightarrow$ (b): it uses Benz\'ecri's compactness and the continuity dependence of the curvature function. \\

Note that our proof here for Theorem \ref{Rank3} bypasses Benz\'ecri compactness and the continuity dependence of the curvature function, only uses Wang's equation itself.

\section{Complete maximal surfaces in $\mathbb H^{2,n}$ and bounded quartic differentials}\label{MaximalSurface}
\subsection{Maximal surfaces in $\mathbb H^{2,n}$}
In this section, we discuss the relation between bounded quartic differentials and complete maximal surfaces in $\mathbb H^{2,n}$. First, we review some backgrounds on maximal surfaces in $\mathbb H^{2,n}$ and their relationships with $SO_0(2,n+1)$-Higgs bundles  developed in \cite{CTT}.

Let $E$ be an $(n+3)$-dimensional real vector space equipped with a signature $(2,n+1)$ quadratic form $Q$. The pseudo-hyperbolic space is defined by
$$\mathbb H^{2,n}:=\{x\in E, Q(x)=-1\}/\{\pm Id\}. $$
Then the group $SO(2,n+1)$ isometrically acts on $\mathbb H^{2,n}$. A spacelike surface in $\mathbb H^{2,n}$ is an immersion of a connected 2-dimensional manifold into $\mathbb H^{2,n}$ whose induced metric is positive definite.

Let $\Sigma$ be a Riemann surface with fundamental group $\pi_1$, $\ti{\Sigma}$ be its universal cover. Let $\rho:\pi_1\rightarrow SO_0(2,n+1)$ be a representation. Let $f:\ti{\Sigma}\rightarrow \mathbb H^{2,n}$ be a spacelike, conformal and $\rho$-equivariant immersion. Consider the trivial bundle $\ti{\Sigma}\times \mb{R}^{2,n+1}$. At each point $p\in \ti{\Sigma}$, the immersion $f$ gives a decomposition $\mb{R}^{2,n+1}=\ti{T}_{f(p)}\oplus\ti{l}_{f(p)}\oplus\ti{N}_{f(p)}$, where $\ti{T}_{f(p)}$ is the tangent space of $f(\ti{\Sigma})$ at $f(p)$, $\ti{l}_{f(p)}$ is the position line $\mb{R}\cdot f(p)$ and $\ti{N}_{f(p)}$ is the normal space of $f(\ti{\Sigma})$ in $T_{f(p)}\mathbb H^{2,n}$ at $f(p)$. Since $f$ is $\rho$-equivariant, $\pi_1$ also acts on and preserves the decomposition $\ti{\Sigma}\times \mb{R}^{2,n+1}=\ti{T}\oplus\ti{l}\oplus\ti{N}$. Then the decomposition of the vector bundle over $\ti{\Sigma}$ descends to $\Sigma$, denoted as $T\oplus l\oplus N.$ Notice that $\ti{T}$ is $\pi_1$-equivariant isomorphic to $\ti{l}\otimes T\ti{\Sigma}$. So we construct a vector bundle over $\Sigma$, which is $T\oplus l\oplus N\cong(l\otimes T\Sigma)\oplus l\oplus N.$

Notice that the standard connection $D$ on $\mb{R}^{2,n+1}$ also descends to $T\oplus l\oplus N$ as
$$D=\left(\begin{smallmatrix}\nabla_T&B&\Pi^{\dagger}\\
B^{\dagger}&\nabla_l& 0\\
\Pi&0&\nabla_N\end{smallmatrix}\right).$$ Here $\Pi\in \Omega^1(\Sigma, \text{Hom}(T, N))$ is called the second fundamental form, $\dagger$ means the adjoint with respect to the corresponding inner product.

\begin{df}
(1) A spacelike surface in $\mathbb H^{2,n}$ is called complete if its induced metric is complete.\\
(2) A spacelike surface in $\mathbb H^{2,n}$ is called maximal if $\tr_{g}\Pi=0$ where $g$ is the induced metric.
\end{df}
\begin{rem}
(1) By \cite[Proposition 3.10]{labourie2020plateau}, a complete maximal surface is an entire graph.\\
(2) In the case when
\end{rem}
Now we relate maximal surfaces in $\mathbb H^{2,n}$ to $SO_0(2,n+1)$-Higgs bundles. Let $K$ be the canonical line bundle of $\Sigma.$
\begin{df}
An $SO_0(2,n+1)$-Higgs bundle over a Riemann surface $\Sigma$ is a tuple $(\mathcal U, q_{\mathcal U}, \mathcal V, q_{\mathcal V}, \eta)$ where
\begin{itemize}
    \item $\mathcal U, \mathcal V$ are respectively of rank 2 and of rank $(n+1)$ holomorphic vector bundles over $\Sigma$ with trivial determinant line bundle $\wedge^2\mathcal U\cong \mathcal O, \wedge^{n+1}\mathcal V\cong \mathcal O.$
    \item $q_{\mathcal U}, q_{\mathcal V}$ are non-degenerate holomorphic sections of $\Sym^2(\mathcal U^*)$ and $\Sym^2(\mathcal V^*)$.
    \item $\eta$ is a holomorphic section of $\Hom(\mathcal U, \mathcal V)\otimes K$
\end{itemize}
\end{df}
Given an $SO_0(2,n+1)$-Higgs bundle $(\mathcal U, q_{\mathcal U}, \mathcal V, q_{\mathcal V}, \eta)$, the associate $SL(n+3,\mathbb C)$-Higgs bundles is $(\mathcal E, \phi)$ by setting $\mathcal E=\mathcal U\oplus \mathcal V$ and $\phi=\begin{pmatrix}0&\eta^{\dagger}\\\eta&0\end{pmatrix}:\mathcal U\oplus \mathcal V\rightarrow (\mathcal U\oplus \mathcal V)\otimes K$, where $\eta^{\dagger}=q_{\mathcal U}^{-1}\circ \eta^*\circ q_{\mathcal V}\in H^0(\Hom(\mathcal V,\mathcal U)\otimes K).$

A Higgs bundle is conformal if the corresponding harmonic map is conformal. By the work in \cite{CTT}, a maximal conformal $SO_0(2,n+1)$-Higgs bundle is determined by $(I,\mathcal V_0, q_{\mathcal V_0}, \beta)$:
$$(\mathcal U,q_{\mathcal U},\mathcal V, q_{\mathcal V}, \eta)=(IK\oplus IK^{-1},\begin{pmatrix}0&1\\1&0\end{pmatrix}, I\oplus \mc{V}_0, \begin{pmatrix}1&0\\0&q_{\mathcal V_0}\end{pmatrix}, \begin{pmatrix} 1&0\\0&\beta\end{pmatrix}),$$
where $I$ is a holomorphic line bundle satisfying $I^2=\mathcal O$, $\mathcal V_0$ is a holomorphic vector bundle of rank $n$ satisfying $\wedge^n\mathcal V_0=I$, $q_{\mathcal V_0}$ is a non-degenerate holomorphic section of $\Sym^2(\mathcal V_0^*)$ and $\beta\in H^0(\Hom(IK^{-1}, \mathcal V_0)\otimes K)$. The original definition for ``maximal" was for compact Riemann surface of genus $g\geq 2.$ Here we adopt this definition for general (possibly non-compact) Riemann surfaces.

Written in the form of $SL(n+3,\mathbb C)$-Higgs bundle, the conformal maximal $SO(2,n+1)$-Higgs bundle above becomes
\begin{equation}
E=IK\oplus IK^{-1}\oplus I\oplus \mc{V}_0,\quad
\phi=\left(\begin{smallmatrix}0&0&0&\beta^\dagger\\
0&0&1&0\\
1&0&0&0\\
0&\beta&0&0
\end{smallmatrix}\right),\label{Higgsbundle}
\end{equation}
where $\beta^\dagger=\beta^*\circ q_{\mc{V}_0}$.
Consider the Hermitian metric solving the Hitchin equation of the form
$$h=\diag(h_{IK}, h_{IK}^{-1}, h_I, h_{\mc{V}_0}).$$
Then the Hitchin equation simplifies to
\begin{eqnarray*}
&&F_{h_{IK}}+\beta^\dagger\wedge (\beta^\dagger)^{*_h}+1^{*_h}\wedge 1=0\\
&&F_{h_{\mc{V}_0}}+\beta\wedge (\beta)^{*_h}+(\beta^\dagger)^{*_h}\wedge \beta^\dagger=0.
\end{eqnarray*}
Let $\ti{\Sigma}$ be a maximal surface in $\mb{H}^{2,n}$. From the discussion above, we construct a vector bundle $T\oplus l\oplus N$ over $\Sigma$. Let $g_T, g_l, g_N$ be the induced metric correspondingly. Consider the complexification of the vector bundle, $T^{\mb{C}}\oplus l^{\mb{C}}\oplus N^{\mb{C}}$. Let $g_T^{\mathbb C}, g_l^{\mathbb C}, g_N^{\mathbb C}, h_T, h_l, h_N$ be the complex linear extension and Hermitian extension of $g_T, g_l, g_N$. The data of Higgs bundles are collected as follows. Notice that $T\cong l\otimes T\Sigma$, $T^{\mb{C}}\Sigma\cong \overline{K^{-1}}\oplus K^{-1}$, and $\overline{K^{-1}}\cong K$ by the Hermitian metric.
\begin{itemize}
\item $(l^{\mathbb C}, (\nabla_{l}^{\mb{C}})^{(0,1)}, h_l)$ gives $(I, \bar{\partial}_{I}, h_I)$.
\item ($T^{\mathbb C}, (\nabla_{T}^{\mb{C}})^{(0,1)}, g_T^{\mathbb C}, h_T)$ gives the orthogonal Hermitian bundle
$$(IK\oplus IK^{-1}, (\bar{\partial}_{I}\otimes\bar{\partial}_{K})\oplus(\bar{\partial}_{I}\otimes\bar{\partial}_{K^{-1}}), \begin{pmatrix}0&1\\1&0\end{pmatrix}, h_{IK\oplus IK^{-1}}).$$
\item $(N^{\mathbb C}, (\nabla_{N}^{\mb{C}})^{(0,1)},g_N^{\mathbb C}, h_N)$ gives the orthogonal Hermitian bundle $(\mc{V}_0, \bar{\partial}_{\mc{V}_0}, q_{\mc{V}_0}, h_{\mc{V}_0})$.

\item The $(1,0)$-part of the second fundamental form $\Pi$ gives $\beta.$
\end{itemize}
The maximality of surfaces implies the Higgs field $\phi$ is holomorphic. Moreover the $(1,0)$-part of the shape operator $B$ gives $\beta^\dagger$, and the standard connection $D^{\mathbb C}=\nabla_h+\phi+\phi^{*_h},$ where $\nabla_h$ is the Chern connection.
Denote by $q_4$ the holomorphic quartic differential $\beta^\dagger\beta=q_{\mathcal V_0}(\beta,\beta)$. By the straight calculation, we obtain that $\tr(\phi^j)=0$ if $j\neq 0(\text{mod} 4)$ and $\tr(\phi^{4j})=4q_4^j$. Therefore, $\beta^\dagger\beta$ only captures the spectral data of $\phi$.

We may consider the geometric objects from the $\rho$-equivariant maximal surface $f:\ti{\Sigma}\rightarrow \mb{H}^{2,n}$, such as the pullback metric $g_f$ and the curvature $k_{\ti{\Sigma}}$ of $g_f$. By the $\rho$-equivariance, they can all be descended to $\Sigma$. We abuse the notation $g$ and $k$ for short if there is no confusion.

We can also build the theory above starting from a maximal surface $X$ as an immersed submanifold in $\mb{H}^{2,n}$. The complex structure of $X$ is from its induced metric. For simplicity, now we consider that $X$ is topologically the disc $\mb{D}$ and omit the representation. In this case the position line bundle $I$ is trivial.

\begin{prop} (\cite[Proposition 4.5 and Equation (20)]{LT})\label{CurvatureBound}
Let $X$ be a maximal surface in $\mathbb H^{2,n}$. Let $g$, $k$ be the induced metric and its curvature. Then
\begin{eqnarray}
&&k=-1+|\beta|^2_h=-1+\frac{1}{2}|\Pi|_g^2\geq -1.\label{Gauss equation}\\
&&\triangle_gk\geq k(1+k).\label{CurvatureEquationReal}\end{eqnarray}
\end{prop}
\begin{rem}
The original proof in \cite{LT} uses the Gauss equation of a maximal surface and the Bochner formula. Here we derive these two equations in terms of Higgs bundles for its own interest as follows. Once one makes the explicit  correspondence between the standard connection between the Higgs bundle and harmonic metric, these two methods are essentially the same.
\end{rem}
\begin{proof}
Suppose locally with respect to a holomorphic frame $e_1, e_2, \cdots, e_{n+3}$ of $E$ where $e_1,e_2, e_3$ are frames of $IK,IK^{-1},I$ respectively which are related by $e_1\cdot\partial/\partial z=e_3=e_2dz$, and $\{e_4, \cdots, e_{n+3}\}$ is a frame of $\mc{V}_0$. Then $\phi=fdz, \beta=\gamma dz$, so $$f=\left(\begin{smallmatrix}0&0&0&\gamma^{\dagger}\\
0&0&1&0\\
1&0&0&0\\
0&\gamma&0&0
\end{smallmatrix}\right),\quad f^{*_h}=\left(\begin{smallmatrix}
0&0&1^{*_h}&0\\
0&0&0&\gamma^{*_h}\\
0&1^{*_h}&0&0\\
(\gamma^{\dagger})^{*_h}&0&0&0\end{smallmatrix}\right).$$ Here $1:IK\rightarrow I$ means to contract $\frac{\partial}{\partial z}$. Then $|1|^2_h=h(\frac{\partial}{\partial z},\frac{\partial}{\partial z})$. So the induced Hermitian metric is $h=|1|^2_hdz\otimes d\bar{z}$, and the induced Riemannian metric is $g=h+\bar{h}=2|1|^2_h|dz|^2$.

We apply Lemma \ref{SimpsonEstimate} to a local holomorphic section $s_1=\left(\begin{smallmatrix}0&0&0&0\\
0&0&0&0\\
1&0&0&0\\
0&0&0&0
\end{smallmatrix}\right).$ Since locally,
\[[f^{*_h},s_1]=\left(\begin{smallmatrix}1^{*_h}\circ 1&&&\\
&0&&\\
&&-1\circ 1^{*_h}&\\
&&&0\end{smallmatrix}\right),\quad
[f,s_1]=\left(\begin{smallmatrix}
0&0&0&0\\
1\circ 1&0&0&0\\
0&0&0&-1\circ \gamma^{\dagger}\\
0&0&0&0
\end{smallmatrix}\right),
\]
then we have
\begin{equation}\label{HitchinEquation1}
\begin{array}{cccc}
\partial_z\partial_{\bar z}\log |1|_h^2=\frac{|[s_1,f^{*_h}]|_h^2-|[s_1,f]|_h^2}{|s_1|_h^2}
=\frac{2|1|_h^4-|1|_h^4-|1|_h^2\cdot|\gamma^{\dagger}|_h^2}{|1|_h^2}=|1|_h^2-|\gamma|_h^2.
\end{array}
\end{equation}
We explain why it is equality here but not inequality. Notice that in \ref{SimpsonEstimate} the inequality only happens at $|h(\partial_{z,h}s,s)|^2\leq |\partial_{z,h}s|_h^2|s|^2_h.$ So we only need to show $(\partial_{z,h})_{\frac{\partial}{\partial z}}s=\lambda s$ for some function $\lambda$ on $X$. Recall $\partial_{z,h}$ is the $(1,0)$ part of the Chern connection $\nabla_{h}$ on $E=IK\oplus IK^{-1}\oplus I\oplus \mc{V}_0$. Notice that both the complex structure and the Hermitian metric are diagonal, so the Chern connection $\nabla_{h}$ and then $\partial_{z,h}$ are also diagonal. Since $IK$ and $I$ are both $1$-dimensional, $\Hom(IK,I)$ is also $1$-dimensional. So $(\partial_{z,h})_{\frac{\partial}{\partial z}}s=\lambda s$ for some function $\lambda$ on $X$.

Since $g=2Re(|1|_h^2dz\otimes d\bar z)$, we obtain \[k=-\frac{2}{2|1|^2_h}\partial_z\partial_{\bar z}\log 2|1|^2_h=-1+\frac{|\gamma|_h^2}{|1|_h^2}=-1+|\beta|^2_h.\]
Since $\beta$ is the $(1,0)$-part of the second fundamental form $\Pi$, we have $|\beta|_h^2=\frac{1}{2}|\Pi|_g^2$. Therefore we prove the first equation.

We apply Lemma \ref{SimpsonEstimate} to a local holomorphic section $s_2=\left(\begin{smallmatrix}0&0&0&0\\
0&0&0&0\\
0&0&0&0\\
0&\gamma&0&0
\end{smallmatrix}\right).$ Since locally,
\[[f^{*_h},s_2]=\left(\begin{smallmatrix}0&&&\\
&\gamma^{*_h}\gamma&&\\
&&0&\\
&&&-\gamma\gamma^{*_h}\end{smallmatrix}\right),\quad
[f,s_2]=\left(\begin{smallmatrix}
0&\gamma^\dagger \gamma&0&0\\
0&0&0&0\\
0&0&0&0\\
0&0&-\gamma\circ 1&0
\end{smallmatrix}\right),
\]
then we have
\begin{eqnarray}\label{HitchinEquation2}
\partial_z\partial_{\bar z}\log |\gamma|_h^2&\geq& \frac{|[s_2,f^{*_h}]|_h^2-|[s_2,f]|_h^2}{|s_2|_h^2}\nonumber\\
&=&\frac{2\text{tr}(\gamma\gamma^{*_h}\gamma\gamma^{*_h})-|\gamma^\dagger\gamma|_h^2-|\gamma\circ 1|_h^2}{|\gamma|_h^2}\\
&=&\frac{2|\gamma|_h^4-|\gamma^\dagger\gamma|_h^2-|1|_h^2\cdot|\gamma|_h^2}{|\gamma|_h^2}\nonumber\\
&\geq&|\gamma|_h^2-|1|_h^2\nonumber,
\end{eqnarray}
where the last inequality follows from the Cauchy-Schwarz inequality $|\gamma^\dagger\gamma|_h\leq |\gamma|_h\cdot|\gamma^\dagger|_h=|\gamma|_h^2$.

Combining Equations (\ref{HitchinEquation1}) and (\ref{HitchinEquation2}), we obtain locally
\begin{equation*}
\partial_z\partial_{\bar z}\log\frac{|\gamma|_h^2}{|1|_h^2}\geq 2(|\gamma|_h^2-|1|_h^2).
\end{equation*}

Globally, we obtain
\begin{equation*}
\triangle_g\log \frac{|\gamma|_h^2}{|1|_h^2}=\frac{1}{2|1|_h^2}\partial_z\partial_{\bar z}\log\frac{|\gamma|_h^2}{|1|_h^2}\geq (\frac{|\gamma|_h^2}{|1|_h^2}-1)=k.
\end{equation*}

Then away from zeros of $\gamma$, we have
\begin{equation*}
\triangle_gk=\triangle_g\frac{|\gamma|_h^2}{|1|_h^2}\geq \frac{|\gamma|_h^2}{|1|_h^2}\cdot\triangle_g\log \frac{|\gamma|_h^2}{|1|_h^2}=k(1+k).
\end{equation*}
Since both sides of the above equation are continuous, the equation holds on the whole surface.
So we prove the second equation.
\end{proof}

The following theorem is the curvature rigidity theorem in \cite[Theorem A]{LT}. With the curvature formula above, we give a different proof.
\begin{prop}\label{CurvatureNegative}
Let $X$ be a complete maximal surface in $\mathbb H^{2,n}$. Then the intrinsic curvature $k$ satisfies either $k<0$ or $k\equiv 0$.
Equivalently, in terms of Higgs bundles, $|\beta|_h^2< 1$ or $|\beta|_h^2\equiv 1$.
\end{prop}
\begin{proof}
Applying Lemma \ref{KEYLEMMA} to Equation (\ref{CurvatureEquationReal}) in  Proposition \ref{CurvatureBound}, we obtain the statement.
\end{proof}

\begin{rem}\label{CompactnessContinuity}
Consider $\mathcal M(n)$ as the set of pairs $(x,X)$ where $X$ is a complete maximal surface in $\mathbb H^{2,n}$ and $x$ is a point on $X$. The natural action of $SO_0(2,n)$ on the space of $\mathcal M(n)$, equipped with the Hausdorff topology is compact, proved in \cite{labourie2020plateau}. Moreover, the curvature $k_X(x)$ depends continuously on the pair $(x,X)\in \mathcal M(n)$.
The original proof in \cite{LT} for Proposition \ref{CurvatureNegative} relies on the cocompactness.
\end{rem}

Next we recover the following theorem shown by Labourie-Toulisse in \cite{LT}.
\begin{thm}(Labourie-Toulisse
\cite{LT})\label{MaximalSurfacePart1}
For a complete maximal surface $X$ in $\mathbb H^{2,n}$, the followings are equivalent:\\
(a) the induced metric has curvature bounded above by a negative constant;\\
(b) the induced metric is conformally bounded with respect to the hyperbolic metric.
\end{thm}
\begin{proof}
For the ``only if" direction, applying Lemma \ref{KEYLEMMA} to Equation (\ref{CurvatureEquationReal}) in  Proposition \ref{CurvatureBound}, we have $k\leq -\delta$ for a positive constant $\delta.$

For the ``if" direction, since the induced metric has curvature satisfying $-1\leq k\leq -\delta$, then by Lemma \ref{KC}, $g_{\mathbb D}\leq g\leq Cg_{\mathbb D}$ for some constant $C=C(\delta)>0$.
\end{proof}
\begin{rem}\label{Discussion}
The above proof from (b) to (a) is the same as the one in \cite{LT}, while the proof from (a) to (b) is different from the one in \cite{LT}. Let's briefly explain the original proof from (b) to (a) in Labourie-Toulisse \cite{LT}. Their proof is by adding the following two equivalent conditions, which is their main goal:\\
(c) the induced metric is Gromov hyperbolic; \\
(d) $\Sigma$ is quasiperiodic, that is, the orbit closure $\overline{SO_0(2,n+1)\cdot (x,X)}$ in $\mathcal M(n)$ does not contain the Barbot surface, in which case the induced metric is flat.\\
The proof of (b)$\Rightarrow$ (a) is by showing (b)$\Rightarrow$(c)$\Rightarrow$(d)$\Rightarrow$(a).


(b) $\Rightarrow$ (c) follows from that quasi-isometry preserves Gromov hyperbolicity.

(c) $\Rightarrow$ (d) follows from the fact that the limit of a sequence of Gromov $\delta$-hyperbolic spaces is still Gromov $\delta$-hyperbolic.

(d) $\Rightarrow$ (a) uses the compactness as in Remark \ref{CompactnessContinuity}.

Note our proof here for Theorem \ref{MaximalSurfacePart1} bypasses the cocompactness , but directly from the equation itself.
\end{rem}

For a maximal surface $X$ in $\mathbb H^{2,n}$, we associate a holomorphic quartic differential $q_4=g_N(\beta, \beta)$. Equivalently, we can define $q_4=\frac{\tr(\phi^4)}{4}$ for the corresponding Higgs bundle $(E,\phi)$.
\begin{thm}\label{EquivalenceDetail}
For a complete maximal surface $X$ in $\mathbb H^{2,n}$, suppose $X$ is conformal to $\mathbb D$. The followings are equivalent:\\
(i) the quartic differential is bounded with respect to the hyperbolic metric;\\
(ii) the induced metric has curvature bounded above by a negative constant;\\
(iii) the induced metric is  conformally bounded with respect to the hyperbolic metric.
\end{thm}
\begin{proof}
By Theorem \ref{MaximalSurfacePart1}, (ii) and (iii) are equivalent.

From (iii) to (i), we have $$|q_4|_{g}=|q_4|_{h}=|\beta^{\dagger}\beta|_h\leq |\beta|_{h}|\beta^{\dagger}|_h=|\beta|^2_h.$$ Then from Equation (\ref{Gauss equation}) and Proposition \ref{CurvatureNegative}, we obtain $|\beta|^2_h\leq 1$. From (iii), $g$ is conformally bounded by the hyperbolic metric, so $q_4$ is also bounded with respect to the hyperbolic metric.

From (i) to (iii), suppose the quartic differential $q_4=\frac{\tr(\phi^4)}{4}$ is bounded with respect to the hyperbolic metric $g_0$. Then the spectrums of the Higgs field $\phi$ are bounded. From Proposition 3.12 in \cite{LiMochizuki}, $|\phi|_{h,g_0}\leq C$. Since the Hermitian metric $h$ on $E$ is diagonal, from (\ref{Higgsbundle}) we have $|\phi|_{h,g_0}^2=2|\beta|_{h,g_0}^2+2|1|_{h,g_0}^2$. So $|1|_{h,g_0}^2\leq C$. Notice that $|1|_{h,g_0}^2=\frac{|f|^2_h|dz|^2_{g_0}}{|dz|^2_g}$, where $f$ is position vector of $X$ in $\mb{H}^{2,n}$, so $|f|_h=1.$ Hence we obtain $g\leq Cg_{0}$ for some constant $C$. For the other direction, from Equation (\ref{Gauss equation}), $k\geq -1$. Then $g\geq g_0$ from Lemma \ref{KC}.
\end{proof}

\subsection{An analogue of the universal Teichm\"uller space}

In the work of F. Labourie and J. Toulisse, they introduce an analogue of the universal Teichm\"uller space $\mathcal{QS}_n$ and define a natural map from $\mathcal{QS}_n$ to the product of the universal Teichm\"uller space and the space of bounded quartic differentials. Let's briefly explain their construction in \cite{LT}.

Let $V$ be a $2$-dimensional real vector space. For any quadruple of pairwise distinct points $(x,y,z,w)$ in $\mathbb P(V)$, denote by $[x,y,z,w]$ its cross ratio. Recall that a homeomorphism $\phi$ of $\mathbb P(V)$ is quasisymmetric if there exist constants $A$ and $B$ greater than $1$ such that for any quadruple of pairwise distinct points in $\mathbb P(V)$, we have if $A^{-1}\leq |[x,y,z,t]|\leq A$, then $B^{-1}\leq |[\phi(x), \phi(y), \phi(z),\phi(t)]|\leq B$.  Let $QS_0$ be the group of quasisymmetric homeomorphisms of $\mathbb P(V)$. The group $PSL(V)$ acts on $QS_0$ by post-composition and the quotient is defined to be the universal Teichm\"uller space, denoted by $\mathcal T(\mathbb H^2)$.

The Einstein universe is the quadric associated to $q$:
\[\partial_{\infty}\mathbb H^{2,n}:=\{x\in \mathbb P(E), q(x)=0\}. \]
The group $SO_0(2,n+1)$ acts transitively on $\partial_{\infty}\mathbb H^{2,n}$. There is a natural generalization of the definition of cross ratio on $\partial_{\infty}\mathbb H^{2,n}$. As a generalization, for a map $\xi$ from $\mathbb P(V)$ to $\partial_{\infty}\mathbb H^{2,n}$ and a pair of constants $A, B$ greater than $1$, the map $\xi$ is $(A,B)$-quasisymmetric if it is positive and for all quadruples $(x,y,z,t)\in\mathbb P(V)^4$, we have if $A^{-1}\leq |[x,y,z,t]|\leq A$, then $B^{-1}\leq |[\xi(x), \xi(y), \xi(z), \xi(w)]|\leq B$.


Let ${QS}_n$ be the space of quasisymmetric maps from $\mathbb P(V)$ to $\partial_{\infty}\mathbb H^{2,n}$ equipped with the $C^0$ topology. The quotient $\mathcal{QS}_n:=QS_n/SO_0(2,n+1)$ is a Hausdorff topological space. Note that $\mathcal{QS}_0$ is $\mathcal T(\mathbb H^2)$. The space $\mathcal{QS}_n$ should be viewed as an analogue of the universal Teichm\"uller space.

By \cite[Theorem B]{LT}, a maximal surface $\Sigma$ is quasiperiodic if and only if the boundary map is quasisymmetric.
By \cite[Theorem 7.1]{LT}, for any element $\xi\in \mathcal{QS}_n$, there is a unique reparametrization $\xi_{repar}\in \mathcal{QS}_n$ of the image of $\xi$. Moreover, suppose the image of $\xi$ bounds a maximal surface whose induced curvature is bounded above by $-c$, then $\xi_{repar}$ is $(A, B)$-quasisymmetric, for a pair constant $(A,B)$ only depending on $c$. The space $\mathcal T(\mathbb H^2)$ acts on $\mathcal{QS}_n$ by post-composition. They define a continuous map $\pi_{\mathbb H^2}:\mathcal {QS}_n\rightarrow \mathcal T(\mathbb H^2)$ by setting $\pi_{\mathbb H^2}(\xi)=\phi$, where $\phi$ is such that $\xi\circ\phi=\xi_{repar}$.

When $\Sigma$ is quasiperiodic, by Theorem 6.3 in \cite{LT}, the uniformization gives a biLipschitz map between $\mathbb H^2$ and $\Sigma$, so $q_4$ is bounded with respect to the hyperbolic metric. Denote by $H^0_b(\mathbb H^2, K^4)$ the vector space of holomorphic quartic differentials on $\mathbb H^2$ that are bounded with respect to the hyperbolic metric.

 In sum, we obtain a map
\begin{eqnarray*}
\mathcal H_{\mathbb H^2}: \mathcal {QS}_n&\rightarrow &\mathcal T(\mathbb H^2)\times H_b^0(\mathbb H^2, K^4)\\
\xi&\mapsto& (\pi_{\mathbb H^2}(\xi), q(\xi)).
\end{eqnarray*}

As a question by F. Labourie and J. Toulisse, we are going to show the map is proper.

\begin{thm}
The map $\mathcal{H}_{\mathbb H^2}$ is proper.
\end{thm}
\begin{proof}
Suppose $(\phi, q)\in\mathcal T(\mathbb H^2)\times H_b^0(\mathbb H^2, K^4)$ satisfies that $\phi$ is $(A, B)$-quasisymmetric and $|q|_{g_{\mathbb D}}\leq M$. It is enough to show $\xi\in\mathcal H_{\mathbb H^2}^{-1}(\phi,q)$ is $(A'', B'')$-quasisymmetric, where $(A'',B'')$ only depends on $A, B, M$. The properness then follows from the action of $PSL(V)\times SO_0(2,n+1)$ on the space of $(A,B)$-quasisymmetric maps is cocompact, proved in Theorem 3.12 in \cite{LT}..

From $|q|_{g_{\mathbb D}}\leq M$, then the induced curvature on $\Sigma$ is bounded from above by a constant $-c$ for $c=c(M)$ from Theorem \ref{EquivalenceDetail}. Thus $\Sigma$ is quasiperiodic by Theorem B in \cite{LT}, see the discussions in Remark \ref{Discussion}.
By Theorem 7.1 in \cite{LT}, since $(x,\Sigma)$ is a quasiperiodic surface whose curvature is bounded by $-c$, then $\xi_{repar}$ is $(A', B')$-quasisymmetric, where $(A',B')=(A',B')(c)>1$.

Since $\phi$ is $(A,B)$-quasisymmetric, then $\phi^{-1}$ is also $(A, B)$-quasisymmetric. Then we have $\xi=\xi_{repar}\circ \phi^{-1}$ is of $(A'',B'')$-quasisymmetric, where $(A'',B'')$ only depends on $A, B, M$. Therefore we finish the proof of properness of $\mathcal H_{\mathbb H^2}$.
\end{proof}

\section{$J$-holomorphic curve in $\mathbb H^{4,2}$ and bounded sextic differentials}\label{G2}
In this section, we discuss the relation between bounded sextic differentials with $J$-holomorphic curves in $\mathbb H^{4,2}.$

We consider the split octonians $\mathbb O'$, whose automorphism group $G_2'=Aut(\mathbb O')\subset SO_0(3,4)$ is the split real $G_2'$. The imaginary split octonians $Im\mathbb O'$ endowed with a natural quadratic form $q$ identifies with $\mathbb R^{3,4}.$ On $\mathbb R^{3,4},$ there is a cross product given by $x\times y:=Im(xy).$ It induces an almost complex structure $J$ on the pseudosphere $$S^{2,4}=\{x\in\mathbb R^{3,4}|q(x,x)=1\}$$ by $$J(X):=x\times X,\quad \text{for all~} x\in S^{2,4}, X\in T_xS^{2,4}\cong x^{\perp}.$$
Consider $(S^{2,4},J)$ with metric multiplies by $-1$, it coincides with the pseudo-hyperbolic space $\mathbb H^{4,2},$
where $$\mathbb H^{4,2}=\{x\in \mathbb R^{4,3}|q(x,x)=-1\}.$$

In an almost complex manifold $(M, J)$, a $J$-holomorphic curve is an immersed surface $\Sigma$ whose tangent bundle $T\Sigma\subset TM$ is preserved by $J$.

Baraglia in \cite{g2Geometry} discoverd that a subcyclic rank $7$ Higgs bundle in the Hitchin section over a domain $\Omega\subset\mathbb C$ together with a harmonic metric $h=diag(h_1,h_2,h_3,1,h_3^{-1}, h_2^{-1},h_1^{-1})$ satisfying $h_1=2h_2h_3$ gives rise to a $J$-holomorphic (also called almost-complex) curve $\nu:\Omega\rightarrow \hat S^{2,4}$. The explicit relation between the resulting surface and Higgs bundles has been further developed in Evans \cite{Parker} using harmonic sequence.

For a $J$-holomorphic curve, the osculation line is defined to be the $J$-complex line in a normal space formed by the images of second fundamental form. Nie in  \cite{NieXin} showed  that subcyclic Higgs bundles in the Hitchin section together with a real harmonic metric $diag(h_1, h_2, h_3, 1, h_3^{-1}, h_2^{-1}, h_3^{-1})$ satisfying $h_1=2h_2h_3$ are characterized as space-like $J$-holomorphic curves in $\mathbb H^{4,2}$ with nowhere vanishing second fundamental form and timelike osculation lines. One can retrieve the holomorphic sextic differential $q$ from the data of the structure equation of the immersion.

The induced Hermitian metric on the $J$-holomorphic curve is
$h=|1|^2_hdz\otimes d\bar z=h_3^{-1}dz\otimes d\bar z$, and the induced Riemannian metric is $g = h +\bar h=2h_3^{-1}(dx^2+dy^2)$(see \cite[Section 3.1]{Parker}).
 \begin{df}
 We call a space-like $J$-holomorphic curve in $\mathbb H^{4,2}$ complete if its induced metric is complete.
 \end{df}
 \begin{rem} (1) It is not clear if the completeness condition implies the surface is a proper embedding or entire.
\\
(2) The complete metric does not necessarily come form a complete real solution for the variant Toda system in Theorem \ref{varianttoda1}.\\
(3) Also, a complete real solution in Theorem \ref{varianttoda1} for the variant Toda system is not obviously satisfying $h_1=2h_2h_3.$ We believe this is true. By the uniqueness, it is true for compact surface case. But for noncompact surfaces, one needs to show it.
 \end{rem}

\begin{lem}\label{EquationG2}
Let $\Sigma$ be a space-like $J$-holomorphic curve in $\mathbb H^{4,2}$ with nowhere vanishing second fundamental form and timelike osculation line. Let $g, k$ be the induced metric and its curvature. Then $k\geq -1.$
If $\Sigma$ is complete, then
\begin{equation}
 \triangle_g k\geq 3k(k+1).
\end{equation}
\end{lem}
\begin{proof}
The surface data corresponds to the Higgs bundle $s(0,\cdots, 0,q_6,0)$ together with a diagonal harmonic metric $$diag(h_1,h_2,h_3,1,h_3^{-1},h_2^{-1},h_1^{-1})$$ satisfying $h_1=2h_2h_3.$
Let $g_0=h_3^{-1}(dz\otimes d\bar z)$.
Since $g=2Re(g_0)$, the curvature $k=\frac{1}{2}K_{g_0}$ and $\triangle g=\frac{1}{2}\triangle_{g_0}.$ By assumption, $g$ is complete. We use the notion in Section \ref{EquationSystem}. Let  $h_i=e^{w_i}g_0^{i-4}$. Set $f_0=e^{w_1+w_2}|q|_{h_3^{-1}}^2, f_1=e^{-w_1+w_2}, f_2=e^{-w_2+w_3}, f_3=e^{-w_3}.$ The associated variant Toda system is
\begin{eqnarray*}
&&\triangle_{g_0}\log f_0=2f_0-f_2+k\\
&&\triangle_{g_0}\log f_1=2f_1-f_2+k\\
&&\triangle_{g_0}\log f_2=2f_2-f_0-3f_1+k
\end{eqnarray*}
From the condition $h_1=2h_2h_3$, we have $f_3=2f_1=1$.

So $k=f_2-1\geq -1.$
And
\begin{eqnarray*}
\triangle_g k&=&\frac{1}{2}\triangle_{g_0}k=\triangle_{g_0} f_2\geq f_2\cdot\triangle_{g_0}\log f_2\\
&=&f_2(2f_2-f_0-\frac{3}{2}+k)\\
&&\text{using $f_0\leq f_1=\frac{1}{2}$ from Lemma \ref{easycase} and the completeness of $g$}\\
&\geq&f_2(2f_2-2+k)=3k(k+1).
\end{eqnarray*}
\end{proof}

Similar to the case of maximal surfaces in $\mathbb H^{2,4},$ we show the following result.
\begin{thm}
For a complete space-like $J$-holomorphic curve $\Sigma$ in $\mathbb H^{4,2}$ with nowhere vanishing second fundamental form and timelike osculation line, its induced curvature is either strictly negative or constantly zero.

Let $q$ be its associated holomorphic sextic differential. Suppose $\Sigma$ is conformal to $\mathbb D$. The following statements are equivalent.
\begin{enumerate}
\item $q$ is bounded with respect to the hyperbolic metric;
\item the induced metric on $\Sigma$ is conformally bounded with respect to the hyperbolic metric;
\item the induced curvature on $\Sigma$ is bounded from above by a negative constant.
\end{enumerate}
\end{thm}
\begin{proof}
 The first statement follows from applying Lemma \ref{KEYLEMMA} to the equation of curvature  $k$ and $k\geq -1$ in Lemma \ref{EquationG2}.

Next we show the equivalence. The equivalence between (1) and (2) follows from Theorem \ref{UpperBound} and Lemma \ref{KC}.

From (2) to (3), applying Lemma \ref{KEYLEMMA} to the equation of $k$ in  Lemma \ref{EquationG2}, we have $k\leq -\delta$ for a positive constant $\delta.$

From (3) to (1), since the induced metric has curvature satisfying $-1\leq k\leq -\delta$, then by Lemma \ref{KC}, $g_{\mathbb D}\leq g\leq Cg_{\mathbb D}$ for some constant $C=C(\delta)>0$.
\end{proof}

%

%
%
%

\bibliographystyle{amsalpha}
\bibliography{bib}

\end{document}